\documentclass[reqno]{amsart}
\usepackage{latexsym, amssymb,hyperref,amssymb,amsfonts}
\usepackage{color,graphicx,mathrsfs,empheq,enumerate, cite}
\usepackage[usenames,dvipsnames]{xcolor}

\newtheorem{theorem}{Theorem}[section]
\newtheorem{proposition}[theorem]{Proposition}
\newtheorem{lemma}[theorem]{Lemma}
\newtheorem{corollary}[theorem]{Corollary}

\newtheorem{definition}[theorem]{Definition}
\newtheorem{assumption}[theorem]{Assumption}
\newtheorem{remark}[theorem]{Remark}
\newtheorem{example}[theorem]{Example}
\theoremstyle{remark}

\makeatletter
\def\l@section{\@tocline{1}{0pt}{1pc}{}{}}
\def\l@subsection{\@tocline{2}{0pt}{1pc}{4.6em}{}}
\def\l@subsubsection{\@tocline{3}{0pt}{1pc}{7.6em}{}}
\renewcommand{\tocsection}[3]{%
  \indentlabel{\@ifnotempty{#2}{\makebox[2.3em][l]{%
    \ignorespaces#1 #2.\hfill}}}#3}
\renewcommand{\tocsubsection}[3]{%
  \indentlabel{\@ifnotempty{#2}{\hspace*{2.3em}\makebox[2.3em][l]{%
    \ignorespaces#1 #2.\hfill}}}#3}
\makeatother

\newcommand{\R}{{\mathord{\mathbb R}}}
\newcommand{\Rd}{{\mathord{\mathbb R}^d}}
\newcommand{\loc}{{\rm loc}}
\newcommand{\id}{{\mathop{\rm \mathbf{id} }}}
\newcommand{\grad}{\nabla}
\newcommand{\wsto}{\stackrel{*}{\rightharpoonup}}
\newcommand{\la}{\left\langle}
\newcommand{\ra}{\right\rangle}

\def\P{{\mathcal P}}
\def\epsilon{\varepsilon}
\newcommand{\argmin}{\operatornamewithlimits{argmin}}

\newcommand{\bxi}{\boldsymbol{\xi}}
\newcommand{\bt}{\mathbf{t}}
\newcommand{\bmu}{\boldsymbol{\mu}}
\newcommand{\bnu}{\boldsymbol{\nu}}

\newcommand{\tomega}{\tilde{\omega}}
\newcommand{\tf}{\tilde{f}}
\newcommand{\tF}{\tilde{F}}

\begin{document}
\title[Nonconvex gradient flow in the Wasserstein metric]{Nonconvex gradient flow in the Wasserstein metric and applications to constrained nonlocal interactions}
\author{Katy Craig}
\address{Department of Mathematics, University of California, Santa Barbara, CA}
\email{kcraig@math.ucsb.edu}
\thanks{This work was supported by a UC President's Postdoctoral Fellowship and U.S. National Science Foundation grant DMS 1401867.}
\date{December 22, 2015}

\begin{abstract}

Over the past fifteen years, the theory of Wasserstein gradient flows of convex (or, more generally, semiconvex) energies has led to advances in several areas of partial differential equations and analysis. In this work, we extend the well-posedness theory for Wasserstein gradient flows to energies satisfying a more general criterion for uniqueness, motivated by the Osgood criterion for ordinary differential equations. We also prove the first quantitative estimates on convergence of the discrete gradient flow or \emph{JKO scheme} outside of the semiconvex case. We conclude by applying these results to study the well-posedness of constrained nonlocal interaction energies, which have arisen in recent work on biological chemotaxis and congested aggregation.
\end{abstract}

\maketitle

\section{Introduction}

For a range of physical and biological processes---from vortex motion in superconductors to biological swarming---the evolution of a large number of individual agents can be modeled as a \emph{gradient flow in the Wasserstein metric}. Heuristically, a curve in the space of probability measures $\mu(t): [0,T] \to \P(\Rd)$  is the Wasserstein gradient flow of an energy $E: \P(\Rd) \to \R \cup \{+\infty \}$ if
\[ \partial_t \mu(t) = - \grad_{W_2} E(\mu(t)) , \]
for a generalized notion of gradient $\grad_{W_2}$ induced by the metric. Provided that the curve $\mu(t)$ and energy $E$ are sufficiently regular, the gradient flow of $E$ corresponds to the partial differential equation
\[ \partial_t \mu = \grad \cdot \left(\mu \grad \frac{\partial E}{\partial \mu} \right) . \]

The utility of Wasserstein gradient flow in studying such partial differential equations was first demonstrated by Otto in his work on the porous medium equation \cite{Otto}, in which he used the Wasserstein perspective to obtain sharp estimates on asymptotic behavior of solutions. Over the past fifteen years, the theory of Wasserstein gradient flows of \emph{convex} and, more generally, \emph{semiconvex} energies has been well-developed---where $E$ is \emph{semiconvex} if there exists $\lambda \in \R$ so that, for any Wasserstein geodesic $\mu_\alpha \in \P(\Rd)$ with $\alpha \in [0,1]$,
\[E(\mu_\alpha) \leq (1-\alpha) E(\mu_0) + \alpha E (\mu_1) - \alpha (1-\alpha) \frac{\lambda}{2}W_2^2(\mu_0,\mu_1). \]
When $\lambda =0$, this reduces to the standard notion of convexity.

In their book on gradient flows in metric spaces, Ambrosio, Gigli, and Savar\'e offer a comprehensive treatment of the well-posedness of Wasserstein gradient flows in the semiconvex case and describe many of the partial differential equations that may be treated within this framework \cite{AGS}. They also prove sharp estimates on the convergence of a time-discretization of the gradient flow, often known as the \emph{discrete gradient flow} or \emph{JKO scheme} \cite{JKO}, which is useful for constructing and approximating solutions to the gradient flow. A common strategy in studying properties of gradient flows is to first prove a property holds at the discrete time level and then show that it is preserved in the passage to continuous time (c.f. \cite{AlexanderKimYao, AGS, MRS}).

In contrast to the case of convex energies, the theory of Wasserstein gradient flow for \emph{nonconvex} (or, more precisely, non-semiconvex) energies is comparatively undeveloped. Such energies arise in a range of applications in physics and biology, including models of vortex motion in superconductors (c.f. \cite{AmbrosioSerfaty, AmbrosioMaininiSerfaty}), biological chemotaxis (c.f. \cite{BlanchetCalvezCarrillo, BlanchetCarlenCarrillo, CarrilloLisiniMainini}), and biological swarming (c.f. \cite{BertozziLaurentRosado, CraigTopaloglu, CraigKimYao}). On one hand, existence of generalized notions of gradient flow and convergence of the discrete gradient flow scheme can still be shown in the nonconvex case \cite{AGS}. On the other hand, uniqueness and stability of solutions have only been considered on a case by case basis and estimates on the rate of convergence of the discrete gradient flow scheme have remained open, in spite of the fact that such estimates can be an essential ingredient in semiconvex gradient flow (c.f. \cite{AlexanderKimYao}).

In addition to energies that fall outside the scope of well-posedness results on convex gradient flow, there has also been interest in energies that, while convex, actually possess better stability properties than their convexity would indicate. Such energies arise, for example, in models of granular media, and Carrillo, McCann, and Villani demonstrated how the the notion of a \emph{modulus of convexity} $\omega(x)$ could be used to prove superior contraction inequalities than were available in the convex theory \cite{CarrilloMcCannVillani}. 

The goal of the present work is to unify and extend the existing theory of nonconvex Wasserstein gradient flow.  Our approach is strongly influenced by the uniqueness results proved by Ambrosio and Serfaty \cite{AmbrosioSerfaty} and Carrillo, Lisini, and Mainini \cite{CarrilloLisiniMainini} for nonconvex, nonlocal interaction energies with Newtonian repulsion or attraction. Following their terminology, we introduce a more general notion of convexity, which we call \emph{$\omega$-convexity}. In particular, we say that $E$ is $\omega$-convex if there exists a constant $\lambda_\omega \in \R$ and a nonnegative, continuous, nondecreasing \emph{modulus of convexity} $\omega(x)$, which vanishes only at $x=0$, so that for any Wasserstein geodesic $\mu_\alpha \in \P_2(\Rd)$ with $\alpha \in [0,1]$,
\[E(\mu_\alpha) \leq (1-\alpha) E(\mu_0) + \alpha E (\mu_1) -  \frac{\lambda_\omega}{2} \left[ (1-\alpha) \omega(\alpha^2 W_2^2(\mu_0,\mu_1)) + \alpha \omega( (1-\alpha)^2 W_2^2(\mu_0,\mu_1)) \right] .\]
When $\omega(x) = x$, this reduces to the standard notion of semiconvexity. A sufficient condition for $\omega$-convexity, provided that $\alpha \mapsto E(\mu_\alpha)$ is sufficiently regular, is that $E$ satisfies an ``above the tangent line'' inequality:
\begin{align*}
 E(\mu_1) - E(\mu_0) - \left. \frac{d}{d \alpha} E(\mu_\alpha) \right|_{\alpha = 0} \geq \frac{\lambda_\omega}{2} \omega \left( W^2_{2}(\mu_0,\mu_1) \right) . \end{align*}
 (See Proposition \ref{omega convex sufficient}.)

Our generalization of the Wasserstein gradient flow theory from the semiconvex to $\omega$-convex case is motivated by the generalization of the theory of ordinary differential equations
\[ \dot{x}(t) = \omega(x(t)) , \quad x:[0,T] \to \R , \quad \omega:\R \to \R , \]
from functions $\omega(x)$ that are Lipschitz to those satisfying Osgood's criterion. Semiconvex Wasserstein gradient flow corresponds to the choice $\omega(x) = x$, in which case  the above ordinary differential equation is well-posed by classical Cauchy-Lipschitz theory. More generally, if there exists a nonnegative, continuous, nondecreasing function $\tomega(x)$, which vanishes only at $x=0$, satisfies $\int_0^1 \frac{dx}{\tomega(x)} = +\infty$, and for which we have
\[ |\omega(x) - \omega(y)| \leq \tomega(|x-y|) \quad \text{ for all } x, y \in \R , \]
then $\omega(x)$ satisfies \emph{Osgood's criterion}, and the above ODE is well-posed locally in time (c.f. \cite[Corollary II.6.2]{Hartman}).

The existence and uniqueness of solutions to this ordinary differential equation are key elements in our analysis of the corresponding Wasserstein gradient flows. Consequently, an essential assumption in the present work is that the modulus of convexity $\omega(x)$ satisfies Osgood's criterion. We additionally require that $\tomega(x)$ decays quickly as $x$ approaches zero---specifically that $\tomega(x) = o(\sqrt{x})$. (In fact, such decay is a consequence of Osgood's criterion whenever $\lim_{x \to 0}\tomega(x)/\sqrt{x}$ exists. See Remark \ref{decay remark}.)

Three key examples of functions $\omega(x)$ satisfying these assumptions are
\begin{enumerate}[(i)]
\item Lipschitz modulus of convexity:  $\omega(x) = \tomega(x) = x$;
 \item polynomial modulus of convexity: for $p\geq 0$, 
 \[ \omega(x) = \begin{cases} x^{p+1} &\text{ for }0 \leq x \leq 1, \\ 1 &\text{ for } x >1, \end{cases} \quad \text{ and } \quad \tomega(x) = x ; \]
\item log-Lipschitz modulus of convexity: \label{log Lipschitz omega} \begin{align*}
 \omega(x) = \tomega(x) = \begin{cases} x |\log x| & \text{ if } 0 \leq x \leq e^{-1-\sqrt{2}} , \\ \sqrt{x^2+ 2(1+\sqrt{2})e^{-1-\sqrt{2}}x} & \text{ if }x > e^{-1-\sqrt{2}}.  \end{cases} \end{align*}
\end{enumerate}
In section \ref{main assumptions and examples}, we provide several examples of energies that are $\omega$-convex for these choices of $\omega(x)$, including the Keller-Segel energy of bounded densities (Example \ref{Newt att ex}), potential energies from models of granular media (Example \ref{potential granular}), and the Chapman-Rubinstein-Schatzman energy for vortex motion in superconductors (Example \ref{vortex example}).

Given an $\omega$-convex energy, for an Osgood modulus of convexity satisfying $\tomega(x) = o(\sqrt{x})$, we prove our main result: that its Wasserstein gradient flow is well-posed and can be approximated by the discrete gradient flow, with quantitative rate of convergence. We prove the gradient flow is unique by showing a contraction  inequality quantifying its stability. In particular, for the specific examples of $\omega(x)$ described above, we obtain the following rates of contraction/expansion (see Example \ref{rate of contraction example}):
\begin{enumerate}[(i)]
\item Lipschitz modulus of convexity:  $W_2(\mu(t),\nu(t)) \leq e^{-\lambda_\omega t} W_2(\mu(0),\nu(0))$ for all $t \geq 0$;
 \item polynomial modulus of convexity: $W_2(\mu(t),\nu(t)) \leq W_2(\mu(0),\nu(0))(1+2 \lambda_\omega pt W_2^{2p}(\mu(0),\nu(0)))^{-1/2p}$ \\
 for $0 \leq W_2(\mu(0),\nu(0)) \leq 1$ and $t\geq0$ sufficiently small;
\item log-Lipschitz modulus of convexity: \label{log lip contr} $W_2(\mu(t), \nu(t)) \leq W_2(\mu(0),\nu(0))^{1/e^{-2 \lambda_\omega t}}$ \\
for $0 \leq W_2(\mu(0),\nu(0)) \leq e^{-1-\sqrt{2}}$ and $t \geq 0$ sufficiently small.
\end{enumerate}
Our rate of contraction for a Lipschitz modulus of convexity, i.e. a semiconvex energy, coincides with the rate obtained by Ambrosio, Gigli, and Savar\'e \cite[Theorem 4.0.4]{AGS}, and our rate of contraction for a polynomial modulus of convexity coincides with the rate shown by Carrillo, McCann, and Villani in the case $\lambda_\omega \geq 0$ \cite[Theorem 1]{CarrilloMcCannVillani}. Our result also provides a rate of contraction for a polynomial modulus of convexity when $\lambda_\omega <0$, which was previously unknown. Likewise, our rate of contraction for the log-Lipchitz case coincides with the rate of contraction obtained by Carrillo, Lisini, and Mainini \cite[Theorem 3.1]{CarrilloLisiniMainini} in the specific case of the bounded Keller-Segel energy with $\lambda_\omega <0$. Our results also extend to a log-Lipschitz modulus of convexity when $\lambda_\omega > 0$.

Along with these results on the stability of the gradient flow, we also estimate the rate at which the discrete gradient flow approximation $\mu^n_{t/n}$ converges to the gradient flow $\mu(t)$. For the three examples of $\omega(x)$ listed above, we obtain the following rates of convergence:
\begin{enumerate}[(i)]
\item Lipschitz and polynomial modulus of convexity:  $W_2(\mu^n_{t/n}, \mu(t)) =O(n^{-1/4})$; \label{lip rate}
 \item  log-Lipschitz modulus of convexity: \label{log lip rate} $W_2(\mu^n_{t/n},\mu(t)) = \mathcal{O} \left( \left[  n^{-1/2} \log(n) \right]^{1/2e^{2\lambda_\omega^- t}}  \right)$.
\end{enumerate}
(The rates of convergence obtained for the Lipschitz and polynomial moduli of convexity coincide due to the fact that our rates depend on the function $\tomega(x)$, and $\tomega(x) = x$ in both cases.) While the estimates for a Lipschitz modulus of convexity are not optimal (see instead \cite[Theorem 4.0.4]{AGS}), we believe our result is still of interest, since it provides the first quantitative estimates on convergence of the discrete gradient flow for energies that are not semiconvex---including energies with merely a log-Lipschitz modulus of convexity.

We conclude by applying our results to gradient flows of constrained nonlocal interaction energies of the form
\begin{align*}
\mathcal{W}_p(\mu) &:= \begin{cases}
\frac{1}{2} \int \mu(x) W*\mu(x) dx &\text{if $\| \mu\|_p \leq C_p$}, \\ +\infty & \text{otherwise,}  \end{cases}
\end{align*}
along with related nonconvex drift diffusion energies.
In the specific case that $p=+\infty$, such energies have arisen in previous work on bounded solutions to the Keller-Segel equation \cite{CarrilloLisiniMainini}, vortex motion in superconductors \cite{AmbrosioSerfaty}, and Gamma convergence of regularized interaction energies \cite{CraigTopaloglu}. For general $p$, these energies relate to previous work on well-posedness of $L^p$ solutions to the aggregation equation \cite{BertozziLaurentRosado}. (See Example \ref{Lp agg example}.) Our particular motivation for studying these types of energies comes from a constrained interaction model, with $W = \Delta^{-1}$ and $p=+\infty$, studied by the author, Kim, and Yao \cite{CraigKimYao}.  In fact, it was the absence of quantitative estimates on convergence of the discrete gradient flow for this specific energy that inspired the development of the theory of nonconvex Wasserstein gradient in the present work.

Our paper is organized as follows. In section \ref{section 2}, we begin by recalling fundamental properties of the Wasserstein metric and define the gradient flow and discrete gradient flow of $\omega$-convex energies. In section \ref{section 3}, we use a Crandall and Liggett type argument to quantify the rate of convergence for the discrete gradient flow, generalizing previous work by the author from the semiconvex case \cite{Craig}. We then use this convergence result to conclude that the gradient flow of an $\omega$-convex energy is well-posed and obtain some of its basic properties, including a contraction inequality quantifying its stability. In section \ref{applicationssection}, we apply these results to gradient flows of constrained nonlocal interaction energies and nonconvex drift diffusion energies. Finally, in the appendix, we discuss how our results can be extended to gradient flows of energies that vary in time---an application that was also motivated by work with Kim and Yao \cite{CraigKimYao}.
 
There are several directions for future work. First, it would be interesting to investigate the sharpness of the assumptions we place on the energy functional. While we believe the assumption that the modulus of convexity $\omega(x)$ is Osgood is necessary, the well-posedness of the gradient flow might still be obtained without the assumption that $\tomega(x) = o(\sqrt{x})$. On the other hand, we don't believe that quantitative estimates on the rate of convergence of the discrete gradient flow can be obtained without this rate of decay. A second direction for future work would be to identify further examples of energy functionals that satisfy our notion of $\omega$-convexity, in addition to those examples described in sections \ref{section 2} and \ref{applicationssection}. In particular, it may be possible to show that, for certain choices of initial data, the gradient flows of even less well-behaved energies remain in a set for which the energy, when restricted to this set, is $\omega$-convex. A third direction for future work would be to further develop the extension of our results to time-dependent energies. While we describe one possible extension in the appendix, such energies arise in a variety of applications (c.f. \cite{CraigKimYao, PetrelliTudorascu}), and sharp conditions to ensure well-posedness of their gradient flows would be a useful addition to the existing theory.

\section{Wasserstein gradient flow and $\omega$-convexity} \label{section 2}

\subsection{Wasserstein metric, geodesics, and generalized geodesics}

We begin by recalling key properties of the Wasserstein metric. For further background, we refer the reader to the books by Ambrosio, Gigli, and Savar\'e \cite{AGS} and Villani \cite{Villani}.

Let $\P(\Rd)$ denote the set of probability measures on $\Rd$. If a measure $\mu$ is absolutely continuous with respect to Lebesgue measure ($\mu \ll \mathcal{L}^d$), we will slightly abuse notation and identify the measure with its density: $d\mu(x) = \mu(x)dx$. For example, we write $\|\mu\|_{L^p} < +\infty$ if $d\mu(x) = \mu(x) dx$ and $\mu(x) \in L^p(\Rd)$.

Given $\mu, \nu \in \P(\Rd)$, a measureable function $\bt : \Rd \to \Rd$ \emph{transports $\mu$ onto $\nu$} if $\nu(B) = \mu(\bt^{-1}(B))$ for all measurable $B \subseteq \Rd$. In this case, $\nu$ is \emph{the push-forward of $\mu$ under $\bt$}, and we write $\nu = \bt \# \mu$. The set of \emph{transport plans} from $\mu$ to $\nu$ is
\[ \Gamma(\mu,\nu) = \{ \gamma \in \P(\Rd \times \Rd) : \pi_1 \# \gamma = \mu, \pi_2 \# \gamma = \nu \}, \]
where $\pi_1$ and $\pi_2$ denote the projections onto the first and second components of $\Rd \times \Rd$. The \emph{Wasserstein distance} between $\mu$ and $\nu$ is then given by
\[ W_2(\mu,\nu) = \inf \left \{ \left ( \int_{\Rd \times \Rd} |x - y|^2 d \gamma(x,y) \right)^{1/2} : \gamma \in \Gamma(\mu,\nu) \right \} . \]
If $W_2(\mu,\nu) < +\infty$, the  infimum is attained, and we denote the set of \emph{optimal transport plans} by $\Gamma_0(\mu,\nu)$.

If $\mu$ does not charge sets of $d-1$ dimensional Hausdorff measure, then the optimal transport plan from $\mu$ to $\nu$ is unique and of the form $(\id \times \bt) \# \mu$, where $\id(x) = x$ is the identity transformation. (See Gigli for sharp conditions guaranteeing the existence of such a plan \cite{Gigli}.) The \emph{optimal transport map} $\bt = \bt_\mu^\nu$ is unique $\mu$-almost everywhere, and if there also exists an optimal transport map from $\nu$ to $\mu$, then $\bt_\mu^\nu \circ \bt_\nu^\mu = \id$ $\mu$-almost everywhere. Finally, a measurable map satisfying $\bt \# \mu = \nu$ is optimal if and only if it is cyclically monotone $\mu$-a.e. \cite{McCannExistence}.

A technical issue that arises when working with the Wasserstein distance on $\P(\Rd)$ is that there exist measures that are infinite distances apart. Consequently, for simplicity, we restrict our attention to the set of measures with finite second moment
\[\P_{2}(\Rd) = \left\{ \mu \in \P(\Rd) : \int|x|^2 d\mu < +\infty \right\} , \]
so that $(\P_{2}(\Rd), W_2)$ is a metric space. We only use the second moment condition to ensure finiteness of the Wasserstein metric, and our arguments also extend to the metric space $(\P_{2,\nu_0}(\Rd),W_2)$ where
\[ \P_{2, \nu_0}(\Rd) = \{\mu \in \P(\Rd) : W_2(\mu,\nu_0) \leq +\infty \}. \]
(See previous work by the author \cite{Craig} for further details on this extension.)

Along with its metric structure, $(\P_2(\Rd),W_2)$ is complete, and convergence can be characterized as
\begin{center}
	\begin{tabular}{lcl}
		$W_2(\mu_n,\mu)\rightarrow 0$ & $\Longleftrightarrow$ & $\mu_n\rightarrow\mu$ weak-$*$ in $\P(\Rd)$ and $\int_{\Rd}|x|^2\,d\mu_n(x)\rightarrow\int_{\Rd}|x|^2\,d\mu(x)$.
	\end{tabular}
\end{center}
Furthermore, $(\P_{2}(\Rd), W_2)$ is a \emph{geodesic space} \cite[Theorem 7.2.2]{AGS},  since any two measures $\mu_0, \mu_1 \in \P_2(\Rd)$ are connected by a \emph{geodesic}  $\mu_\alpha \in \P_2(\Rd)$, $\alpha \in [0,1]$, satifying
\[ W_2(\mu_\alpha,\mu_\beta) = |\beta-\alpha| W_2(\mu_0,\mu_1) \text{ for all }\alpha, \beta \in [0,1] . \]
By \cite[Theorem 7.2.2]{AGS}, all geodesics are of the form 
\begin{align*} 
\mu_\alpha = \left( (1-\alpha)\pi_1 + \alpha \pi_2 \right)  \# \gamma , \ \text{ for } \gamma \in \Gamma_0(\mu_0,\mu_1) .
\end{align*}
In particular, if $\bt_{\mu_0}^{\mu_1}$ is the unique optimal transport map, then the geodesic $\mu_\alpha$ from $\mu_0$ to $\mu_1$ is unique and given by $\mu_\alpha = \left( (1-\alpha)\id + \alpha \mathbf{t}_{\mu_0}^{\mu_1} \right) \# \mu_0$.

A recurring difficulty when analyzing gradient flows in the Wasserstein metric is that, unlike in the Hilbert space case, the square metric $\mu \mapsto W_2^2(\nu, \mu)$ is not convex \cite[Example 9.1.5]{AGS}. To circumvent this difficulty, Ambrosio, Gigli, and Savar\'e introduced the notion of \emph{generalized geodesics}  \cite[Lemma 9.2.1, Definition 9.2.2]{AGS}.

\begin{definition} \label{gengeo}
Given $\mu_0, \mu_1, \nu \in \P_{2}(\Rd)$, a generalized geodesic from $\mu_0$ to $\mu_1$ with base $\nu$ is a curve of the form
\[ \mu_\alpha = ((1-\alpha) \pi^1 + \alpha \pi^2 ) \# \bnu , \text{ where } \bnu \in \P(\Rd \times \Rd \times \Rd) , \  \pi^{1,3} \# \bnu \in \Gamma_0(\mu_0,\nu) , \text{ and } \pi^{2,3} \# \bnu \in \Gamma_0(\mu_1,\nu) . \]
\end{definition} 
\noindent Geodesics are a particular type of generalized geodesic, when $\nu$ coincides with either $\mu_0$ or $\mu_1$.

Alongside the notion of generalized geodesics, we will also consider the following type of transport distance \cite[Definition 1.14]{Craig}:
\begin{definition} \label{transport distance def}
For $\nu \in \P_2(\Rd)$ that does not charge sets of $d-1$ dimensional Hausdorff measure, the \emph{$(2, \nu)$-transport metric} $W_{2, \nu}: \P_{2}(\Rd) \times \P_{2}(\Rd) \to \R$ is given by
\begin{align*} 
 W_{2,\nu}&(\mu_0,\mu_1) := \left( \int |\mathbf{t}_{\nu}^{\mu_0} -
\mathbf{t}_{\nu}^{\mu_1} |^2 d \nu \right)^{1/2} .
\end{align*}
More generally, for any $\nu, \mu_0,\mu_1 \in \P_2(\Rd)$,  define
\begin{align*} 
& W_{2,\bnu}(\mu_0,\mu_1) := \left( \int |\pi_1 -
\pi_2 |^2 d \bnu \right)^{1/2} ,
\end{align*}
where $\bnu \in \P(\Rd \times \Rd \times \Rd)$ satisfies $\pi^{1,3} \# \bnu \in \Gamma_0(\mu_0,\nu),  \pi^{2,3} \# \bnu \in \Gamma_0(\mu_1,\nu)$.
\end{definition}
\begin{remark} \label{transport metric bounds W2}
This provides an upper bound for the Wasserstein metric: ${W_{2,\bnu}(\mu_0,\mu_1) \geq W_2(\mu_0,\mu_1)}$.
\end{remark}

The \emph{$(2, \nu)$-transport metric} is an example of a \emph{pseudo-Wasserstein metric}, as introduced by  Ambrosio, Gigli, and Savar\'e \cite[Equation 9.2.5]{AGS}. When $\nu \in \P_2(\Rd)$ does not charge sets of $d-1$ dimensional Hausdorff measure, then $W_{2,\nu}$ is a metric on $\P_{2}(\Rd)$  \cite[Proposition 1.15]{Craig}. More generally, for any fixed $\rho \in \P_{2}(\Rd)$ and $\bnu$ as in Definition \ref{transport distance def}, $\mu \mapsto W^2_{2,\bnu}(\mu,\rho)$ is convex along the generalized geodesic induced by $\bnu$  \cite[Proposition 1.15]{Craig}:
\begin{align} \label{convexitytransportmetric}
W_{2,\bnu}^2(\mu_\alpha, \rho) = (1-\alpha) W_{2,\bnu}^2(\mu_0,\rho) + \alpha W_{2,\bnu}^2(\mu_1,\rho) - \alpha (1-\alpha) W_{2,\bnu}^2(\mu_0,\mu_1) , \quad \forall \alpha \in [0,1] .
\end{align}
This identity is a key element in our proofs of Proposition \ref{discreteEVI} and Theorem \ref{recineqthm}.

\subsection{$\omega$-convexity}
The main goal of the present work is to further develop the theory of gradient flows for energies that satisfy a more general notion of convexity, which we call $\omega$-convexity. This parallels the definition introduced by Carrillo, Lisini, and Mainini \cite[Theorem 4.1]{CarrilloLisiniMainini} for a specific choice of function $\omega(x)$ satisfying $\omega(x) = x|\log x|$, for $x$ small, and $\omega(x) = O(x)$, for $x$ large. (See Example \ref{omega examples} (\ref{log Lipschitz omega}).)
\begin{definition}[notions of convexity] \label{convexitydef}
Given an energy functional $E: \P_{2}(\Rd) \to \R \cup \{+\infty\}$, a curve $\mu_\alpha \in \P_{2}(\Rd)$, and a function $d: \P_2(\Rd) \times \P_2(\Rd) \to [0, +\infty)$,
\begin{enumerate}[(i)]
\item $E$ is \emph{convex} along $\mu_\alpha$ if \label{first convex}
\[ E(\mu_\alpha) \leq (1-\alpha) E(\mu_0) + \alpha E (\mu_1); \]
\item $E$ is \emph{semiconvex} along $\mu_\alpha$ with respect to $d$ if, for some $\lambda \in \R$,
\[E(\mu_\alpha) \leq (1-\alpha) E(\mu_0) + \alpha E (\mu_1) - \alpha (1-\alpha) \frac{\lambda}{2}d(\mu_0,\mu_1)^2;\] \label{second convex}
\item $E$ is \emph{$\omega$-convex} along $\mu_\alpha$ with respect to $d$ if, for some $\omega: [0,+\infty) \to [0,+\infty)$ and $\lambda_\omega \in \R$, \label{omega convex def}
\[E(\mu_\alpha) \leq (1-\alpha) E(\mu_0) + \alpha E (\mu_1) -  \frac{\lambda_\omega}{2} \left[ (1-\alpha) \omega(\alpha^2 d(\mu_0,\mu_1)^2) + \alpha \omega( (1-\alpha)^2 d(\mu_0,\mu_1)^2) \right],\] \label{third convex}
where the \emph{modulus of convexity} $\omega(x)$ is continuous, nondecreasing, and vanishes only at $x=0$. 
\end{enumerate}
These notions of convexity along curves lead to corresponding notions of convexity along geodesics and generalized geodesics in the Wasserstein metric:
\begin{enumerate}[(a)]
\item $E$ is convex/semiconvex/$\omega$-convex along \emph{geodesics} if, for any $\mu_0,\mu_1 \in \P_2(\Rd)$, there exists a geodesic $\mu_\alpha$ from $\mu_0$ to $\mu_1$ so that (\ref{first convex})/(\ref{second convex})/(\ref{third convex}) holds, with $d = W_2$,
\item $E$ is convex/semiconvex/$\omega$-convex along \emph{generalized geodesics} if, for any $\mu_0, \mu_1, \nu \in \P_2(\Rd)$, there exists a generalized geodesic $\mu_\alpha$ from $\mu_0$ to $\mu_1$ with base $\nu$ so that (\ref{first convex})/(\ref{second convex})/(\ref{third convex}) holds, with $d= W_{2,\bnu}$.
\end{enumerate}
\end{definition}

We will often use $\lambda^-= \max\{ 0, -\lambda\}$ to denote the negative part of $\lambda$. Using this notation, the relative strengths of the above convexity assumptions are as follows:
\begin{itemize}
\item Convexity and semiconvexity with $\lambda =0$ are equivalent; furthermore, if $E$ is convex, then $E$ is semiconvex  for all $\lambda \leq 0$.
\item Semiconvexity and $\omega$-convexity with $\omega(x) =x$ are equivalent; furthermore, if $E$ is semiconvex, then $E$ is $\omega$-convex for all $\omega(x) \geq x$ and $\lambda_\omega^- \geq  \lambda^-$.
\item If $E$ is convex/semiconvex/$\omega$-convex along generalized geodesics, then, in particular, it possesses the same convexity property along geodesics.
\end{itemize}

In order to describe fundamental properties of $\omega$-convex energies, we denote the \emph{domain} of an energy by $D(E) = \{ \mu : E(\mu) < +\infty \}$ and say $E$ is \emph{proper} if $D(E) \neq \emptyset$. Likewise, we define \emph{metric local slope} of an energy $E$ by
\[ |\partial E|(\mu) =  \limsup_{\nu \to \mu} \frac{(E(\mu) - E(\nu))^+}{W_2(\mu,\nu)} , \text{ for all }\mu \in D(E). \]

A key property of convex and semiconvex functions is that they satisfy an ``above the tangent line'' inequality, often referred to as an HWI inequality in the context of Wasserstein gradient flow. As long as $\omega(x) = o(\sqrt{x})$ as $x \to 0$, $\omega$-convex functions satisfy an analogous inequality:

\begin{proposition}[HWI inequality] \label{HWIproposition}
If $E$ is $\omega$-convex along geodesics and $\omega(x) = o(\sqrt{x})$ as $x \to 0$, then
for all $\mu_0 \in D(|\partial E|) $
\[ E(\mu_0) - E(\mu_1) \leq |\partial E|(\mu_0) W_2(\mu_0,\mu_1) - \frac{\lambda_\omega}{2} \omega(W^2_{2}(\mu_0,\mu_1)) . \]
\end{proposition}
\noindent The proof of this proposition is identical to the proof in the semiconvex case \cite[Theorem 2.4.9]{AGS}.

Using this inequality, one may also show that if the enegy $E$ is lower semicontinuous, its metric local slope $|\partial E|$ is also lower semicontinuous.
\begin{proposition}[lower semicontinuity of slope] \label{slope lsc}
Suppose $E$ is lower semicontinuous and $\omega$-convex along geodesics, with $\omega(x) = o(\sqrt{x})$ as $x \to 0$. Then $|\partial E|$ is lower semicontinuous.
\end{proposition}

\begin{proof}
By Proposition \ref{HWIproposition},
\begin{align} \label{HWI slope formula} |\partial E|(\mu) = \sup_{\nu \neq \mu} \left( \frac{E(\mu)- E(\nu)}{W_2(\mu,\nu)} + \frac{\lambda_\omega}{2} \frac{ \omega(W_2^2(\mu,\nu))}{W_2(\mu,\nu)} \right)^+ .
\end{align}
Arguing as in \cite[Corollary 2.4.10]{AGS} gives the result.
\end{proof}

In order to verify that a given energy is $\omega$-convex, it is often cumbersome to show Definition \ref{convexitydef} (\ref{omega convex def}) directly. We close this section with a sufficient condition that is often more convenient in practice.

\begin{proposition} \label{omega convex sufficient}
Suppose that for all generalized geodesics $\mu_\alpha$ from $\mu_0$ to $\mu_1$ with base $\nu$ such that $\mu_0 ,\mu_1 \in D(E)$, $E(\mu_\alpha)$ is differentiable for $\alpha \in [0,1]$, $\frac{d}{d \alpha} E(\mu_\alpha) \in L^1([0,1])$, and 
\begin{align*}
 E(\mu_1) - E(\mu_0) - \left. \frac{d}{d \alpha} E(\mu_\alpha) \right|_{\alpha = 0} \geq \frac{\lambda_\omega}{2} \omega \left( W^2_{2, \bnu}(\mu_0,\mu_1) \right) . \end{align*}
Then $E$ is $\omega$-convex along generalized geodesics. Furthermore, if $E$ merely satisfies these assumptions in the specific case that $\nu =\mu_0$ or $\nu=\mu_1$, then $E$ is $\omega$-convex along geodesics.
\end{proposition}

\begin{proof}
Define $\pi_\alpha := (1-\alpha) \pi_1 + \alpha \pi_2$. Fix $\mu_0,\mu_1,\nu \in \P_2(\Rd)$, and let $\mu_\alpha = \pi_\alpha \# \bnu$, where $\pi^{1,2} \# \bnu \in \Gamma_0(\mu_0,\nu)$ and $\pi^{1,3} \# \bnu \in \Gamma_0(\mu_1,\nu)$. Next, define $\bnu_\alpha: = (\pi_\alpha \times \pi_2 \times \pi_3)\# \bnu$ so that $\pi^{1,2} \# \bnu_\alpha \in \Gamma_0(\mu_\alpha, \nu)$ and $\pi^{1,3} \# \bnu_\alpha \in \Gamma_0(\mu_1,\nu)$ and, for any $\beta \in [0,1]$, $\hat{\mu}_\beta:= \pi_{\beta} \# \bnu_\alpha = \mu_{(1-\beta)\alpha + \beta}$ is a generalized geodesic from $\mu_\alpha$ to $\mu_1$ with base $\nu$. By assumption,
\begin{align} \label{suff1}
E(\mu_1) - E(\mu_\alpha) - \left. \frac{d}{d \beta} E(\hat{\mu}_\beta) \right|_{\beta = 0} \geq \frac{\lambda_\omega}{2} \omega \left((1-\alpha)^2 W_{2,\bnu}^2(\mu_0,\mu_1) \right).
\end{align}
Likewise, if we take $\tilde{\bnu}_\alpha := ( \pi_\alpha \times \pi_1 \times \pi_3 ) \# \bnu$, $\tilde{\mu}_\beta := \pi_\beta \# \tilde{\bnu}_\alpha =\mu_{(1-\beta)\alpha}$ is a generalized geodesic from $\mu_\alpha$ to $\mu_0$, and
\begin{align} \label{suff2}
E(\mu_0) - E(\mu_\alpha) - \left. \frac{d}{d \beta} E(\tilde{\mu}_\beta) \right|_{\beta = 0} \geq \frac{\lambda_\omega}{2} \omega \left(\alpha^2 W_{2,\bnu}^2(\mu_0,\mu_1) \right).
\end{align}
Since 
\[ \left. \frac{d}{d \beta} E(\hat{\mu}_\beta) \right|_{\beta = 0} = (1-\alpha) \left.\frac{d}{d \beta} E(\mu_\beta) \right|_{\beta = \alpha} \quad \text{ and } \quad  \left. \frac{d}{d \beta} E(\tilde{\mu}_\beta) \right|_{\beta = 0 } = - \alpha \left. \frac{d}{d \beta} E(\mu_\beta) \right|_{\beta = \alpha} , \]
multiplying  (\ref{suff1}) by $\alpha$,  (\ref{suff2}) by $(1-\alpha)$, and adding them together gives the result,
\begin{align*}
E(\mu_\alpha) &\leq (1-\alpha) E(\mu_0) + \alpha E(\mu_1) - \frac{\lambda_\omega}{2} \left[ (1-\alpha) \omega(\alpha^2 W^2_{2, \bnu}(\mu_0,\mu_1)) + \alpha \omega( (1-\alpha)^2 W^2_{2, \bnu}(\mu_0,\mu_1)) \right].
\end{align*}
\end{proof}

\subsection{Gradient flow and discrete gradient flow of $\omega$-convex energies}

We now define the Wasserstein \emph{gradient flow} and \emph{discrete gradient flow} of an $\omega$-convex energy.
We require the following regularity of our gradient flow in time:
\begin{definition}[locally absolutely continuous curve] \label{abs cts def}
$\mu(t): \R \to \P_{2}(\Rd)$ is \emph{locally absolutely continuous} if for all bounded $I \subseteq \R$, there exists $m \in L^1(I)$ so that $W_2(\mu(t),\mu(s)) \leq \int_s^t m(r) dr$, $\forall s \leq t \in I$.
\end{definition}

We now define what it means for such a curve $\mu(t)$ to be the gradient flow of an energy $E$ with respect to the Wasserstein metric structure.
\begin{definition}[gradient flow] \label{gradflowdef} Suppose $E: \P_{2}(\Rd) \to \R \cup \{ +\infty \}$ is proper, lower semicontinuous, and $\omega$-convex along generalized geodesics.
A locally absolutely continuous  curve $\mu: (0,+\infty) \to \P_{2}(\Rd)$ is a\emph{ gradient flow} of $E$ with initial data $\mu_0 \in D(E)$ if $\mu(t) \xrightarrow{t \to 0} \mu_0$ and
\begin{align} \label{continuous evi}
\frac{1}{2} \frac{d}{dt} W_2^2(\mu(t), \nu) + \frac{\lambda_\omega}{2} \omega( W_2^2( \mu(t), \nu)) \leq E(\nu) - E( \mu(t))  ,  \quad \forall \nu \in D(E) , \text{ a.e. }  t >0 .
\end{align}
\end{definition}

 This definition generalizes Ambrosio, Gigli, and Savar\'e's \emph{evolution variational inequality} characterization of the gradient flow of semiconvex energies \cite[Theorem 4.0.4 (iii)]{AGS}, and it parallels Carrillo, Lisini, and Mainini's notion of gradient flow, which they defined in the specific case that $\omega(x) = x|\log x|$, for $x$ small, and $\omega(x) = O(x)$, for $x$ large \cite[Theorem 3.1 (i)]{CarrilloLisiniMainini}. 

In order to prove that solutions exist, we consider a time discretization of the gradient flow, analogous to the implicit Euler method approximation of gradient flow in Euclidean space. This time discretization was first used in the Wasserstein context by Jordan, Kinderlehrer, and Otto, and consequently is often known as the \emph{JKO scheme} \cite{JKO}.

On one hand, one can show that this scheme converges to a generalized notion of the gradient flow without imposing any convexity assumptions on the energy \cite[Theorem 2.3.1 and 2.3.3]{AGS}. On the other hand, \emph{quantitative} rates of convergence of the discrete gradient flow scheme to the gradient flow have only been obtained for semiconvex energies \cite[Theorem 4.0.4]{AGS}. Since such estimates are often essential in applications of the Wasserstein gradient flow theory in partial differential equations (c.f. \cite{AlexanderKimYao, CraigKimYao}), a main goal of this paper is to quantify the rate of convergence of the discrete gradient flow scheme in a more general setting.

With this motivation in mind, we now turn to the precise definition of the Wasserstein \emph{discrete gradient flow} and  some of its basic properties.
\begin{definition}[discrete gradient flow] \label{discrete gradient flow def}
Given an energy $E:\P_2(\Rd) \to \R \cup \{ +\infty\}$ and $\tau>0$, 
\begin{align*}
  &\text{The \emph{proximal map} $J_\tau$ is given by } J_0 \mu := \mu \text{ and }  J_\tau \mu := \argmin_{\nu \in \P_{2}(\Rd)} \left\{ \frac{1}{2\tau} W_2^2(\mu,\nu) + E(\nu) \right\} \text{ for } \tau >0 . \\
&\text{In addition, denote }   J^n_\tau \mu :=  \underbrace{J_\tau\circ J_\tau \circ \dots \circ J_\tau}_{\text{$n$ times}} \mu .
   \end{align*}
We will refer to any sequence $\mu^n_\tau \in J^n_\tau \mu$ as a \emph{discrete gradient flow sequence} with \emph{initial data} $\mu$ and \emph{time step} $\tau$.
\end{definition}
\noindent In general, the discrete gradient flow sequence is not unique, and we use $\mu^n_\tau$ to denote any such sequence.

In what follows, we consider energy functionals satisfying the following assumption, which ensures a discrete gradient flow sequence exists for all $\tau>0$ sufficiently small:
\begin{assumption} \label{prox nonempty}
Suppose $E$ is proper, lower semicontinuous, and there exists $\tau_*\in(0, +\infty]$ so that for all $\mu \in \P_{2}(\Rd)$ and $0 \leq \tau <\tau_*$, $J_\tau \mu \neq \emptyset$.
\end{assumption}


We conclude this section by recalling fundamental properties of the proximal map and discrete gradient flow.
By definition, for any $\mu \in D(E)$, 
\begin{align} \label{onestep1} W_2^2(\mu,\mu_\tau) \leq 2 \tau (E(\mu) - E(\mu_\tau)) , \ \forall \mu_\tau \in J_\tau \mu .
\end{align}
Furthermore, by \cite[Lemma 3.1.2]{AGS}, for any $\mu \in D(E)$,
\begin{align} \label{basic continuity estimates}
\lim_{\tau \to 0} \sup_{\mu_\tau \in J_\tau \mu} W_2(\mu_\tau, \mu) =0  \text{ and } \lim_{\tau \to 0} \inf_{ \mu_\tau \in J_\tau \mu} E(\mu_\tau) = E(\mu) .
\end{align}
Again, by definition, the energy is nonincreasing along any discrete gradient flow,
\[ E(\mu^n_\tau) \leq E(\mu^{n-1}_\tau) \leq \dots \leq E(\mu) , \]
and, by \cite[Lemma 3.2.2]{AGS}, for any $\mu \in D(E)$, 
there exists a nondecreasing function $C(\mu,\cdot): [0,+\infty) \to [0, +\infty)$ so that for $0\leq  \tau< \tau_*$,
\begin{align} \label{tau mu def}
 E(\mu) - E(\mu^n_\tau) \leq  C(\mu, n \tau) .
\end{align}

By combining the above inequalities, we conclude that for all $0 \leq \tau < \tau_*$,
\begin{align} \label{tau mu cons}
W_2(\mu^{j-1}_\tau,\mu^{j}_\tau) &\leq \sqrt{2 \tau C(\mu,n\tau)} \text{ for all } 1 \leq j \leq n , \\
 W_2(\mu,\mu^n_\tau) &\leq \sum_{j=1}^n W_2(\mu^{j-1},\mu^{j}) \leq \sqrt{n} \left(\sum_{j=1}^n 2 \tau(E(\mu^{j-1}_\tau) - E(\mu^j_\tau)) \right)^{1/2} \leq \sqrt{2 n \tau C(\mu, n \tau)}. \label{base case}
\end{align}

Finally, we will also use the following ``above the tangent line'' inequality for the proximal map \cite[Lemma 10.1.2]{AGS}. Suppose $\mu, \nu \in P_{2}(\Rd)$ and there exists an optimal transport map $\bt_{\mu_\tau}^\mu \# \mu_\tau = \mu$. Then for any other transport map $\bt$ satisfying $\bt \# \mu_\tau = \nu$,
\begin{align}  E(\nu)-E(\mu_\tau) \geq \frac{1}{\tau} \int \la \bt_{\mu_\tau}^{\mu} - \id, \bt - \id \ra d \mu_\tau - \frac{1}{2 \tau} \|\bt - \id \|_{L^2(\mu_\tau)}^2    \label{AGSELeqn} .
\end{align}
(See \cite[Lemma 10.3.4]{AGS} for the general case, in which $\bt_{\mu_\tau}^\mu$ and $\bt$ are replaced by transport plans.)

We close this section with a lemma relating the proximal map $J_\tau$ with large and small time steps. This improves on the analogous result from the author's previous work \cite[Theorem 2.3]{Craig}, as it does not require convexity of the energy. For simplicity of notation, we assume there exist optimal transport maps between all measures. The proof in the general case is analogous.

\begin{lemma}[proximal map with large vs. small time steps] \label{proxmapdifftimelem}
Suppose $E$ satisfies assumption \ref{prox nonempty} and there exist optimal transport maps between all measures. Then for all $\mu \in \P_2(\Rd)$, ${0 \leq h \leq \tau < \tau_*}$, and $\mu_\tau \in J_\tau \mu$,
\[ \mu_\tau \in J_h \left[ \left( \frac{\tau-h}{\tau}
\mathbf{t}_\mu^{\mu_\tau} + \frac{h}{\tau} \id \right) \# \mu \right] . \]
\end{lemma}

\begin{proof}
The result is trivially true for $h =\tau$, so suppose $h<\tau$.
Let $\bxi = \frac{1}{\tau}(\bt_{\mu_\tau}^\mu - \id)$. Since $h/\tau < 1,$
\[ (\id + h \bxi) = \left(\id + \frac{h}{\tau} (\bt_{\mu_\tau}^\mu - \id) \right)  = \left( \frac{\tau-h}{\tau} \id + \frac{h}{\tau} \bt_{\mu_\tau}^\mu \right) \]
is cyclically monotone. Thus, if we define $\nu = (\id + h \bxi) \# \mu_\tau$, by uniqueness of the optimal transport map, $\bt_{\mu_\tau}^\nu = \id + h \bxi$. Furthermore,  $\nu = \left( \frac{\tau-h}{\tau} \bt_{\mu}^{\mu_\tau}+ \frac{h}{\tau}\id \right) \# \mu$, and
it suffices to show $\mu_\tau = \nu_h$ for $\nu_h \in J_h \nu$.

Rearranging, 
\begin{align} \label{proxmapeq1}
\frac{1}{h} (\bt_{\mu_\tau}^\nu - \id) = \bxi = \frac{1}{\tau} (\bt_{\mu_\tau}^\mu - \id) . 
\end{align}
Then, by inequality (\ref{AGSELeqn}), with $\bt = \bt_\nu^{\nu_h} \circ \bt_{\mu_\tau}^\nu$,
\[ E(\nu_h) - E(\mu_\tau) \geq \frac{1}{\tau} \int \la \bt_{\mu_\tau}^\mu - \id , \bt_\nu^{\nu_h} \circ \bt_{\mu_\tau}^\nu - \id \ra d \mu_\tau - \frac{1}{2\tau} \|\bt_\nu^{\nu_h} \circ \bt_{\mu_\tau}^\nu - \id \|_{L^2(\mu_\tau)} . \]
Applying (\ref{proxmapeq1}) and rearranging,
\[ E(\nu_h) - E(\mu_\tau) \geq \frac{1}{h} \int \la \id - \bt_\nu^{\mu_\tau} , \bt_\nu^{\nu_h} - \bt_\nu^{\mu_\tau} \ra d \nu - \frac{1}{2 \tau} W_{2, \nu}^2(\nu_h, \mu_\tau) . \]
Adding $\frac{1}{2h}W_2^2(\nu,\nu_h)$ to both sides and using that $E(\nu_h) + \frac{1}{2h}W_2^2(\nu,\nu_h) \leq E(\mu_\tau) + \frac{1}{2h} W_2^2(\nu,\mu_\tau)$,
\[ \frac{1}{2h}W_2^2(\nu,\mu_\tau) \geq \frac{1}{2h} W_2^2(\nu,\nu_h) + \frac{1}{h} \int \la \id - \bt_\nu^{\mu_\tau} , \bt_\nu^{\nu_h} - \bt_\nu^{\mu_\tau} \ra d \nu - \frac{1}{2 \tau} W_{2, \nu}^2(\nu_h, \mu_\tau) . \]
By the identity $\la a, a-b \ra = |a|^2/2 + |a-b|^2/2 - |b|^2/2$, with $a = \id - \bt_\nu^{\mu_\tau}$ and $b = \id - \bt_\nu^{\nu_h}$,
\[ 0 \geq \frac{1}{2h} W_{2,\nu}^2(\nu_h, \mu_\tau) - \frac{1}{2\tau}W_{2,\nu}^2(\nu_h,\mu_\tau) . \]
Since $h < \tau$, this implies $0 \geq W_{2,\nu}^2(\nu_h,\mu_\tau)  \geq W_2^2(\nu_h, \mu_\tau)$, i.e. $\mu_\tau = \nu_h$.
\end{proof}

\subsection{Main assumptions and examples} \label{main assumptions and examples}
We now turn to the main assumptions we will place on our energy functional $E$.
In order to ensure the notion gradient flow given in Definition \ref{gradflowdef} is well-posed and to quantify the rate of convergence of the discrete gradient flow, we require that the modulus of convexity $\omega(x)$ satisfy Osgood's criterion and decay sufficiently quickly as $x \to 0$. As motivation for these assumptions, we recall Osgood's criterion for uniqueness of ordinary differential equations in Euclidean space.
\begin{definition}[Osgood modulus of continuity] \label{Osgoodcondition} Given $U \subseteq_{\text{open}} \Rd$,
a function $\omega: U \to \Rd$ has an \emph{Osgood modulus of continuity} if $|\omega(x) - \omega(y)| \leq \tilde{\omega}(|x-y|)$ for all $x,y \in U$, where $\tomega: [0,+\infty) \to [0, +\infty)$  is continuous, nondecreasing, vanishes only at $x=0$, and $\int_0^1 \frac{dx}{\tomega(x)} = +\infty$.
\end{definition}
\begin{theorem}[Osgood's criterion for uniqueness of ODEs, {c.f. \cite[Corollary II.6.2]{Hartman}}] Given $U \subseteq_{\text{open}} \Rd$, suppose $\omega: U \to \Rd$ has an Osgood modulus of continuity. Then for any $x \in U$, there exists $\delta >0$ so that the ordinary differential equation $\dot{x}(t) = \omega(x(t))$, $x(0) = x_0,$ has a unique solution for all $-\delta \leq t \leq \delta$.
\end{theorem}

By replacing $\tomega(x)$ with its concave majorant, we may assume that $\tilde{\omega}(x)$ in the definition of Osgood modulus of continuity is concave (c.f. \cite[Lemma 6.1]{DeVoreLorentz}). Furthermore, since  any concave and nonnegative function is subadditive, we have
\begin{align} \label{tomega own modulus} |\tomega(x) - \tomega(y) | \leq \tomega(|x-y|) ,
\end{align}
i.e. $\tomega(x)$ is its own Osgood modulus of continuity.

Inspired by this criterion for uniqueness of ordinary differential equations in Euclidean space, we now introduce the notion of an Osgood modulus of \emph{convexity}.
\begin{definition}[Osgood modulus of convexity] \label{Osgood modulus convexity}
Suppose $\omega:[0, +\infty) \to [0,+\infty)$ is continuous, nondecreasing, and vanishes only at $x=0$. We say $\omega$ is an \emph{Osgood modulus of convexity} if  $\omega$ has an Osgood modulus of continuity in the sense of Definition \ref{Osgoodcondition}.
\end{definition}

In what follows, we will assume that our energies are $\omega$-convex, where $\omega(x)$ is an Osgood modulus of convexity and $\tilde{\omega}(x) = o(\sqrt{x})$ as $x \to 0$. The former ensures that the associated ordinary differential equations
\begin{align} \label{omega ode}
\begin{cases}
\frac{d}{dt} F_t(x) &= \lambda_\omega \omega(F_t(x)) , \\
 \ \  \ F_0(x)&= x .
\end{cases}
\quad \quad \quad \quad
\begin{cases}
\frac{d}{dt} \tF_t(x) &= -\lambda_\omega^- \tomega(F_t(x)) , \\
 \ \  \ F_0(x)&= x .
\end{cases}
\end{align}
are well-posed locally in time. The latter is a sufficient condition for the HWI inequality (Proposition \ref{HWIproposition}), since $\omega(x) \leq \tomega(x)$. It is also an essential element in our proof that the discrete gradient flow is a Cauchy sequence (Theorem \ref{W2RasBound}), which we use to quantify its convergence to a solution of the gradient flow (Theorem \ref{expform}).

\begin{remark} \label{decay remark}
Note that if  $\omega(x)$ is an Osgood modulus of convexity and $\lim_{x \to 0} \tomega(x)/\sqrt{x}$ exists, then we must have $\tomega(x) = o(\sqrt{x})$. Consequently, this additional decay assumption is only needed to prevent oscillations of $\tomega(x)/\sqrt{x}$ as $x \to 0$.
\end{remark}

\begin{example}[moduli of convexity] \label{omega examples}
We list three examples of functions $\omega(x)$ that are Osgood moduli of convexity and for which $\tilde{\omega}(x) = o(\sqrt{x})$ as $x \to 0$. We also list the function $F_t(x)$ that solves the associated ordinary differential equation (\ref{omega ode}).
\begin{enumerate}[(i)]
\item Lipschitz modulus of convexity:  $\omega(x) = \tomega(x) = x$. \\
For all $x\geq 0$ and $t \geq 0$, $F_t(x) = e^{-\lambda_\omega t}x$.\label{Lipschitz omega}
 \item polynomial modulus of convexity: for $p\geq 0$, 
 \[ \omega(x) = \begin{cases} x^{p+1} &\text{ for }0 \leq x \leq 1, \\ 1 &\text{ for } x >1, \end{cases} \quad \text{ and } \quad \tomega(x) = x . \]
 \label{polynomial omega} 
For $0 \leq x \leq 1$ and $0 \leq t < (x^{-p}-1)/ \lambda_\omega^+ p$, $F_t(x) = x (1 - \lambda_\omega p t x^{p})^{-1/p}$. 
\item log-Lipschitz modulus of convexity: \label{log Lipschitz omega} \begin{align*}
 \omega(x) = \tomega(x) = \begin{cases} x |\log x| & \text{ if } 0 \leq x \leq e^{-1-\sqrt{2}} , \\ \sqrt{x^2+ 2(1+\sqrt{2})e^{-1-\sqrt{2}}x} & \text{ if }x > e^{-1-\sqrt{2}}.  \end{cases} \end{align*}
 For $0 \leq x \leq e^{-1-\sqrt{2}}$ and $t< \log \left( \log(x)/(-1-\sqrt{2}) \right)/\lambda_\omega^+$,  $F_t(x) = x^{e^{-\lambda_\omega t}}$. 
\end{enumerate}
\end{example}

For the remainder of this work, we consider energy functionals for which the discrete gradient flow is well-defined and that are $\omega$-convex, with $\omega(x)$ satisfying the above Osgood criterion and decay as $x \to 0$. Specifically, we impose the following assumption:
\begin{assumption}[energy functional] \label{main assumptions}
Suppose $E$ is proper, lower semicontinuous, and $\omega$-convex along generalized geodesics for an Osgood modulus of convexity $\omega(x)$ satisfying $\tilde{\omega}(x) = o(\sqrt{x})$ as $x \to 0$. Furthermore, suppose that there exists $\tau_*\in(0, +\infty]$ so that for all $\mu \in \P_{2}(\Rd)$ and $0 \leq \tau <\tau_*$, $J_\tau \mu \neq \emptyset$.
\end{assumption}

We now turn to several examples of energy functionals satisfying this assumption.

\begin{example}[energies with Newtonian attraction] \label{Newt att ex}
Define
\begin{align} \label{psidef1}
\mathcal{N}(x) := \begin{cases}  \frac{1}{2 \pi} \log(|x|) &\text{ if }d=2, \\
\frac{1}{d(2-d)\alpha_d}|x|^{d-2} &\text{ if } d \neq 2, \end{cases} \quad \quad 
\psi(x) &:= \begin{cases} x (\log x)^2 & \text{ if } 0\leq x \leq e^{-1-\sqrt{2}}, \\ x+ 2(1+\sqrt{2})e^{-1-\sqrt{2}}  & \text{ if }x > e^{-1-\sqrt{2}} , \end{cases} 
 \end{align}
 where $\alpha_d$ is the volume of the $d$-dimensional unit ball. Fix $M>0$ and consider the following energies,
\begin{enumerate}[(i)]
\item Keller-Segel energy of bounded densities  \cite{CarrilloLisiniMainini}: \label{bdd KS}
\[ E(\mu):= \begin{cases}  \frac{1}{2} \int \mu(x) \mathcal{N}*\mu(x) dx + \int \mu(x) \log \mu(x) dx & \text{ if } \|\mu\|_\infty \leq M , \\ + \infty &\text{ otherwise.} \end{cases} \]
\item  congested aggregation via Newtonian interaction \cite{CraigKimYao}:
\[ E(\mu):= \begin{cases} \frac{1}{2} \int \mu(x) \mathcal{N}*\mu(x) dx & \text{ if } \|\mu\|_\infty \leq M , \\ + \infty &\text{ otherwise.} \end{cases} \] \label{cong agg}
\item drift diffusion approximation of congested aggregation \cite{CraigKimYao}: \\
 for $\nu \in L^\infty(\Rd)$ with $\|\nu\|_\infty \leq M$,
\[ E(\mu) := \begin{cases} \int \mathcal{N}*\nu(x) \mu(x) dx + \frac{1}{m-1} \int \mu^m(x) dx &\text{ for $\mu \ll \mathcal{L}^d$,} \\
+\infty &\text{ otherwise} .\end{cases}\] \label{drift approx}
\end{enumerate}
All three of these energies satisfy assumption \ref{main assumptions} with $\lambda_\omega = -C_{d,M} \leq0$ and $ \omega(x) =  \sqrt{x \psi(x)}$.
For $d \geq 2$, this is a consequence of Theorem \ref{constrained interaction theorem} and Proposition \ref{Wexamples prop}. For $d=1$, $\mathcal{N}(x)$ is convex, hence so are all three energies \cite{McCannConvexity}---in particular, they are  $\omega$-convex with $\lambda_\omega = 0$.
\end{example}

\begin{example}[repulsive-attractive interaction energy]
Consider the following repulsive-attractive energy of densities with bounded $L^m(\Rd)$ norm \cite{CraigTopaloglu} for fixed $M>0$:
\begin{align*}
& E(\mu):= \begin{cases}  \frac{1}{2} \int \mu(x) \left( |x-y|^q/q - |x-y|^p/p \right) \mu(y) dx dy & \text{ if } \|\mu\|_m \leq M, \\ + \infty &\text{ otherwise,} \end{cases} \\
&2-d \leq p <0 < q\leq 2 \text{ and }m \geq d/(p+d-2).
 \end{align*}
Theorem \ref{constrained interaction theorem} ensures that this energy satisfies assumption \ref{main assumptions}  in $d \geq 3$ for $\lambda_\omega = -C_{d,M} <0$  and $\omega(x) = \sqrt{x \psi(x)}$, with $\psi(x)$ as in (\ref{psidef1}).
\end{example}

\begin{example}[polynomial potential energies and $\phi$-uniformly convex energies] \label{potential granular}
Consider the following potential energies that arise in models of granular media \cite{CarrilloMcCannVillani}:
\begin{align*}
&E(\mu) = \frac{\beta}{b+2} \int_\Rd |x|^{b+2} d \mu(x) \quad \text{ for }\beta, b \geq 0 .
\end{align*}
By Proposition \ref{Carrillo et al Prop}, these energies are $\omega$-convex along generalized geodesics
for
\[\lambda_\omega =\frac{ 2^{b+2}}{b+2} \text{ and } \omega(x) = \begin{cases}  x^{(b+2)/2} &\text{ for }0 \leq x \leq 1, \\ 1 &\text{ otherwise.} \end{cases} \]

More generally, Proposition \ref{Carrillo et al Prop} ensures that any energy which is $\phi$-uniformly convex, in the sense of Carrillo, McCann, and Villani  \cite[Definition 2, Assumptions ($\phi_0$)-($\phi_2$)]{CarrilloMcCannVillani} and differentiable along geodesics is also $\omega$-convex along  geodesics, for an Osgood modulus of convexity $\omega(x)$ satisfying $\tilde{\omega}(x) = o(\sqrt{x})$ as $x \to 0$.
\end{example}

\begin{example}[vortex motion in superconductors] \label{vortex example}
Ambrosio and Serfaty \cite{AmbrosioSerfaty} considered the gradient flow of the following energy, corresponding to the Chapman-Rubinstein-Schatzman mean-field model for superconductivity \cite{ChapmanRubinsteinSchatzman}:
\begin{align*}
 \Phi_\beta (\mu) = \frac{\beta}{2} \mu(\Omega) + \frac{1}{2} \int_\Omega |\grad h_\mu|^2 + |h_\mu - 1|^2 , \quad \quad  &   \beta \geq 0 , \ \Omega \subseteq \R^2 \text{ smooth and bounded, } \\
&\begin{cases} -\Delta h_\mu + h_\mu = \mu &\text{ in } \Omega , \\
 h_\mu = 1 &\text{ on } \partial \Omega. \end{cases} 
\end{align*}
If this energy is restricted to the set of measures with $\|\mu\|_\infty \leq M$, i.e.,
\[ \Phi_\beta^\infty(\mu) := \begin{cases} \Phi_\beta(\mu) &\text{ if } \|\mu\|_\infty \leq M , \\ +\infty &\text{ otherwise,} \end{cases} \]
then, by \cite[Proposition 3.5]{AmbrosioSerfaty} and Proposition \ref{omega convex sufficient}, $\Phi^\infty_\beta(\mu)$ is $\omega$-convex along geodesics for $\lambda_\omega = - C_{M,\Omega} <0$  and $\omega(x) = \sqrt{x \psi(x)}$, with $\psi(x)$ as in (\ref{psidef1}).

\end{example}

\begin{example}[aggregation equation in $L^p$ spaces] \label{Lp agg example}
Bertozzi, Laurent, and Rosado \cite{BertozziLaurentRosado} showed that solutions of the aggregation equation 
\[ \partial_t \mu(t) + \grad \cdot (\mu(t) v(t)) = 0 , \quad v(t) = - \grad K*\mu(t) , \  \mu: [0,T] \to \P_2(\Rd) , \]
with $\|\mu(t)\|_p \leq M$ for $1< p < +\infty$ and $\grad K \in W^{1,p'}(\Rd)$ for $1/p + 1/p' = 1$, are well-posed on bounded time intervals. These solutions are gradient flows of the height constrained interaction energy restricted to the corresponding $L^p$ space,
\[ E(\mu) = \begin{cases} \frac{1}{2} \int K(x-y) d \mu(x) d \mu(y) &\text{ if } \|\mu\|_p \leq M , \\ + \infty &\text{ otherwise, }\end{cases}  \]
which is $\omega$-convex along generalized geodesics with $\lambda_\omega = -C_{M, \|\grad K\|_{W^{1,p'}}} < 0$ and $\omega(x) = x$. (In particular, these energies are semiconvex with respect to the Wasserstein metric.)
\end{example}

We close this section by reviewing some elementary properties of the ordinary differential equation (\ref{omega ode}) associated to $\omega(x)$.
By separation of variables, when the solution of (\ref{omega ode}) exists, it is given by
\begin{align} \label{ode solution} F_t(x) = \begin{cases} x &\text{ if }\lambda_\omega = 0 , \\ \Phi^{-1}( \Phi(x) + t) &\text{ if } \lambda_\omega \neq 0 , \end{cases}   \text{ where } \Phi:(0, +\infty) \to \R \text{ is defined by } \Phi(x)= \int_1^x \frac{dy}{\lambda_\omega \omega(y)} .
\end{align}
If $\lambda \leq 0$, $F_t(x)$ solves (\ref{omega ode}) for all $t \geq 0$, and if $\lambda >0$, $F_t(x)$ solves (\ref{omega ode}) for $0 \leq t < \Phi(+\infty) - \Phi(x)$. $F_t(x)$ and its spatial inverse $F^{-1}_t(x) = F_{-t}(x)$ are continuous and strictly increasing in $x$. If $\lambda_\omega \leq 0$, $F_t(x)$ is nonincreasing in $t$, and if $\lambda_\omega >0$, $F_t(x)$ is nondecreasing in $t$.

In order to prove that the discrete gradient flow converges, we will make frequent use of a time discretization of the ordinary differential equation (\ref{omega ode}). For $\tau \geq 0$, the function
\begin{align} \label{ftaudef}
 f_\tau(x):= \begin{cases} x + \lambda_\omega\tau \omega(x) &\text{ if } x \geq 0, \\ 0 &\text{ if } x < 0,\end{cases}
 \end{align}
 is the first step of the explicit Euler method for the ordinary differential equation (\ref{omega ode}) and 
\[ f^{(m)}_\tau(x) =  \underbrace{f_\tau\circ f_\tau \circ \dots \circ f_\tau(x)}_{\text{$m$ times}} \]
is the $m$th step of this method.
The function $f_\tau(x)$ plays a role analogous to the function $x \mapsto (1+\lambda \tau)x$ in the semiconvex case \cite{AGS, Craig}. In particular, as $m \to +\infty$,  $f^{(m)}_{t/m}(x)$ converges to $F_t(x)$:

\begin{proposition} \label{odeproposition}
For $x \geq 0$ and
\[ 0 \leq t < T := \begin{cases}  \min\{ \Phi(+\infty)/\lambda_\omega , \Phi(+\infty) - \Phi(x) \} &\text{ if } \lambda_\omega >0, \\+ \infty & \text{ if } \lambda_\omega \leq 0 , \end{cases} \]
suppose $F_t(x)$ is the unique solution of the ordinary differential equation (\ref{omega ode}). For $m$ sufficiently large,
\begin{align} \label{odeerrorest}
|F_t(x) - f^{(m)}_{t/m}(x)| \leq \begin{cases}
 \tF_{ |\lambda_\omega| t} \left(\frac{|\lambda_\omega| t \omega(F_t(x))}{m} \right) &\text{ if } \lambda_\omega > 0 , \\
 \tF_{|\lambda_\omega| t} \left( \frac{|\lambda_\omega| t  \omega(x)}{m}  \right) &\text{ if } \lambda_\omega \leq 0 .
\end{cases}  \end{align}
In particular,  $\lim_{m \to \infty} f^{(m)}_{t/m}(x) = F_t(x)$.
\end{proposition}
\noindent We defer the proof of this proposition to appendix section \ref{ode appendix}.

In what follows, we will make frequent use of the following monotonicity properties of $f_\tau(x)$:

\begin{lemma} \label{monotonicitylem}
Suppose $\omega(x)$ is an Osgood modulus of convexity and $\tilde{\omega}(x) = o(\sqrt{x})$ as $x \to 0$. For any $r \geq1$, define 
$ c_r:= \max_{0 \leq x \leq r} \tilde{\omega}(x)/\sqrt{x}$.
\begin{enumerate}[(i)]
\item $\begin{aligned}[t]
&\text{For } 0 \leq x \leq y  , && f_\tau(x) \leq f_\tau(y) && \text{if } \lambda_\omega \geq 0, \\
&\text{For } 0 \leq x \leq y \leq r, \tau < (c_r |\lambda_\omega|)^{-1},  && f_\tau(x) \leq f_\tau(y) +|\lambda_\omega| \tau \tilde{\omega}\left( \lambda_\omega^2 c_{r}^2 \tau^2 \right) \leq f_\tau(y) + \lambda_\omega^2 c_r^2 \tau^2 && \text{if } \lambda_\omega < 0.
\end{aligned}$ \label{ftau monotone} \vspace{2mm}
\item \label{ftau break}
$\begin{aligned}[t]
&\text{For $x,y \geq 0$, } & \hspace{2.55cm} f_\tau(x+y) &\leq f_\tau(x) + y +\lambda_\omega \tau \tilde{\omega}(y)  &&\hspace{2.55cm} \text{ if }\lambda_\omega  \geq 0 , \\
&& f_\tau(x+y) &\leq f_\tau(x) + y &&  \hspace{2.55cm}\text{ if }\lambda_\omega < 0 .
\end{aligned}$
\end{enumerate}
\end{lemma}
\noindent See appendix section \ref{ode appendix} for the proof.

\section{Convergence of the discrete gradient flow} \label{section 3}
We now turn to our main result: all discrete gradient flow sequences converge to the unique solution of the gradient flow, hence the gradient flow of an $\omega$-convex energy is well-posed. To quantify the rate of convergence of the discrete gradient flow, we follow a similar approach to Crandall and Liggett in their seminal work on nonlinear semigroups on a Banach space \cite{CrandallLiggett}. Such an approach has been previously used in the context of metric space gradient flows of semiconvex energies by Cl\'ement and Desch \cite{ClementDesch} and in the specific case of Wasserstein gradient flows of semiconvex energies in previous work by the author \cite{Craig}.

In section \ref{one step section}, we prove several estimates for one step of the discrete gradient flow, comparing $\mu$ to elements of $J_\tau \mu$. These estimates include a discrete time version of the ``evolution variational inequality'', by which we defined the gradient flow of an $\omega$-convex energy, as well as a contraction inequality for the discrete gradient flow. At the end of this section, we iterate the contraction inequality to estimate the Wasserstein distance between the $n$th elements of two discrete gradient flow sequences in terms of their initial data.

In section \ref{convergence subsection}, we apply these estimates to quantify the rate of convergence of the discrete gradient flow sequence. As in Rasmussen's version of Crandall and Liggett's original proof \cite{R,Y}, the key element is an asymmetric recursive inequality for discrete gradient flow sequences with different time steps. Once we have this, we use an inductive argument to control the Wasserstein distance between the discrete gradient flows. This allows us to conclude that the discrete gradient flow sequence is Cauchy and, since the Wasserstein metric is complete, must converge. Our argument also provides the rate of convergence.

In section \ref{properties section}, we show that the limit of any discrete gradient flow sequence is, in fact, the unique Wasserstein gradient flow of the energy. We then use approximation by the discrete gradient flow sequence to obtain fundamental properties of the gradient flow. In particular, we show that solutions of the gradient flow satisfy a contraction inequality, which ensures their stability under perturbations of the initial data.

\subsection{One-step estimates} \label{one step section}

In this section, we prove two estimates quantifying the behavior of one step of the discrete gradient flow. First,  we prove a discrete time evolution variational inequality in terms of $W_{2,\bmu}$. This generalizes a previous result of the author from the semiconvex case to the $\omega$-convex case \cite[Theorem 1.22]{Craig}. 
\begin{proposition}[discrete EVI] \label{discreteEVI}
Suppose $E$ satisfies Assumption \ref{main assumptions}. Then for all $\mu, \nu \in D(E)$, there exists $\bmu \in \P(\Rd \times \Rd \times \Rd)$ with $\pi^{1,3} \# \bmu \in \Gamma_0(\mu_\tau,\mu)$ and $ \pi^{2,3} \# \bmu \in \Gamma_0(\nu,\mu)$ so that
 \begin{align} \label{discevieqn} f_\tau(W_{2,\bmu}^2(\mu_\tau, \nu)) - W_2^2(\mu,\nu) \leq 2 \tau ( E(\nu) - E(\mu_\tau) ) - W_2^2(\mu, \mu_\tau)  \quad \text{ for } 0 \leq \tau < \tau_*.
\end{align}
\end{proposition}
\begin{proof}
Since $E$ is $\omega$-convex along generalized geodesics, there exists $\bmu$ as above so that $E$ is convex along the curve $\nu_\alpha: = ((1-\alpha)\pi^1 + \alpha \pi^2) \# \bmu$. Combining this with the fact that $\frac{1}{2 \tau} W_2^2(\mu, \cdot)$ is $\frac{1}{2 \tau}$-convex along all generalized geodesics with base $\mu$ \cite[Lemma 9.2.1]{AGS}
\begin{align*}
&\frac{1}{2\tau} W_2^2(\mu,\mu_\tau) + E(\mu_\tau)  \leq \frac{1}{2\tau} W_2^2(\mu,\nu_\alpha) + E(\nu_\alpha) \\
&\quad \leq (1-\alpha) \left[ \frac{1}{2\tau} W_2^2(\mu,\mu_\tau) + E(\mu_\tau) \right] + \alpha \left[ \frac{1}{2\tau} W_2^2(\mu,\nu) + E(\nu) \right]  - \frac{1}{2\tau} \alpha (1-\alpha) W_{2,\bmu}^2(\mu_\tau,\nu)  \\
&\quad \quad - \frac{\lambda_\omega}{2} \left[ (1-\alpha) \omega(\alpha^2 W^2_{2,\bmu}(\mu_\tau,\nu)) + \alpha \omega( (1-\alpha)^2 W^2_{2,\bmu}(\mu_\tau,\nu)) \right]
\end{align*}
Dividing by $\alpha$, multiplying by $2\tau$, and rearranging,
\begin{align*}
&(1-\alpha) W_{2,\bmu}^2(\mu_\tau,\nu)  - W_2^2(\mu,\nu) + \lambda_\omega\tau \left[ \frac{1}{\alpha} (1-\alpha) \omega \left(\alpha^2 W^2_{2,\bmu}(\mu_\tau,\nu) \right) + \omega\left( (1-\alpha)^2 W^2_{2,\bmu}(\mu_\tau,\nu) \right) \right]\\ 
&\quad \leq 2 \tau \left[ E(\nu) -   E(\mu_\tau) \right]  - W_2^2(\mu,\mu_\tau).
\end{align*}
Since $\lim_{\alpha \to 0} \frac{1}{\alpha} \omega(\alpha^2 W^2_{2,\bmu}(\mu_\tau,\nu)) = 0$, sending $\alpha$ to zero gives the result.
\end{proof}

We now apply this result to  prove a contraction inequality for the discrete gradient flow.

\begin{theorem}[contraction inequality] \label{cothm}
Suppose $E$ satisfies assumption \ref{main assumptions}.
Then for $\mu, \nu \in D(E)$,
\begin{align*} f_\tau^{(2)}(W_{2}^2(\mu_\tau, \nu_\tau)) \leq W_2^2(\mu,\nu) + o(\tau) \text{ as } \tau \to 0 . \end{align*}
In particular, 
\begin{align*}
 f_\tau^{(2)}(W_{2}^2(\mu_\tau, \nu_\tau))  &\leq W_2^2(\mu,\nu) + \lambda_\omega \tau \tilde{\omega} \left( 2 \tau  (E(\nu_\tau) - E(\mu_\tau) )   \right) +2 \tau ( E(\mu) - E(\mu_\tau) )   && \text{ if } \lambda_\omega >  0, \\
 f_\tau^{(2)}(W_{2}^2(\mu_\tau, \nu_\tau)) &\leq W_2^2(\mu,\nu) - \lambda_\omega \tau  \tomega\left(R^2 W_2(\nu,\nu_\tau) \right)  +     2 \tau(E(\mu) - E(\mu_\tau)) + 3\lambda_\omega^2 c_r^2\tau^2  && \text{ if } \lambda_\omega \leq 0 ,
\end{align*}
where the first inequality holds for all $0\leq \tau < \tau_*$ and the second inequality holds for 
\begin{align*} 0\leq \tau<  \min \left\{ 1,\tau_*, (c_r |\lambda_\omega|)^{-1}, \frac{1}{2} (E(\mu) - E(\mu_\tau))^{-1}, \frac{1}{2}(E(\nu) - E(\nu_\tau))^{-1} \right\},
\end{align*}
where $R \geq \max \{W_2(\mu,\nu), 3 \}$ and $r :=4( R^2+ |\lambda_\omega| \tomega(R^2)) $.
\end{theorem}

\begin{proof}
By Proposition \ref{discreteEVI},
\begin{align} \label{discrete EVI1}
f_\tau(W_{2, \bmu}^2(\mu_\tau, \nu_\tau)) - W_2^2(\mu,\nu_\tau) &\leq 2 \tau (E(\nu_\tau) - E(\mu_\tau) ) - W_2^2(\mu, \mu_\tau)  ,  \\
f_\tau(W_{2, \bnu}^2(\nu_\tau, \mu)) - W_2^2(\mu,\nu) &\leq 2 \tau ( E(\mu) - E(\nu_\tau) ) - W_2^2(\nu, \nu_\tau) .  \label{discrete EVI2}
\end{align}

First, we consider the case when $\lambda_\omega > 0$. 
Since $W_{2}^2(\mu_\tau, \nu_\tau) \leq W_{2, \bmu}^2(\mu_\tau, \nu_\tau) $, applying Lemma \ref{monotonicitylem} (\ref{ftau monotone}) and (\ref{ftau break}) to inequality  (\ref{discrete EVI1}) gives
\begin{align} f_\tau^{(2)}(W_{2}^2(\mu_\tau, \nu_\tau)) 
&\leq f_\tau \left( W_2^2(\mu,\nu_\tau)\right) +  \lambda_\omega \tau \tilde{\omega} \left( 2 \tau (E(\nu_\tau) - E(\mu_\tau) )   \right) + 2 \tau (E(\nu_\tau) - E(\mu_\tau) ) . \label{discrete EVI6}
 \end{align}
Likewise, since $W_{2}^2(\mu, \nu_\tau) \leq W_{2, \bnu}^2(\mu, \nu_\tau)$, applying Lemma \ref{monotonicitylem} (\ref{ftau monotone}) to inequality (\ref{discrete EVI2}) and adding to  (\ref{discrete EVI6}) gives the result.

We now consider the case $\lambda_\omega \leq 0$, again starting from inequalities (\ref{discrete EVI1}) and (\ref{discrete EVI2}). 
We begin by applying the triangle inequality, inequality (\ref{onestep1}), and the upper bounds on $\tau$ to obtain  crude bounds for $W_{2, \bmu}^2(\mu_\tau, \nu_\tau) $ and $W_{2, \bnu}^2(\nu_\tau, \mu)$, both of which are less than
\begin{align*}
3 \left[ W_2^2(\mu, \mu_\tau) + W_2^2(\nu,\nu_\tau) +  W_2^2(\mu,\nu) \right] \leq 3 (2 + R^2) \leq   4( R^2+ |\lambda_\omega| \tomega(R^2)) =:r .
\end{align*}
 Since $W_{2}^2(\mu_\tau, \nu_\tau) \leq W_{2, \bmu}^2(\mu_\tau, \nu_\tau) \leq r$ and  $W_{2}^2(\nu_\tau, \mu) \leq W_{2, \bnu}^2(\nu_\tau, \mu) \leq r$, by Lemma \ref{monotonicitylem} (\ref{ftau monotone}), we may replace  inequalities (\ref{discrete EVI1}) and (\ref{discrete EVI2}) with
\begin{align}
f_\tau(W_{2}^2(\mu_\tau, \nu_\tau)) - W_2^2(\mu,\nu_\tau) &\leq 2 \tau (E(\nu_\tau) - E(\mu_\tau) ) - W_2^2(\mu, \mu_\tau) +\lambda_\omega^2 c_r^2 \tau^2 , \label{discrete EVI4}  \\
f_\tau(W_{2}^2(\nu_\tau, \mu)) - W_2^2(\mu,\nu) &\leq 2 \tau ( E(\mu) - E(\nu_\tau) ) - W_2^2(\nu, \nu_\tau)  +\lambda_\omega^2 c_r^2 \tau^2 . \label{discrete EVI5}
\end{align}

Adding inequalities (\ref{discrete EVI4}) and (\ref{discrete EVI5}) together and using $f_\tau(x)-x = \lambda_\omega\tau \omega(x)$  gives
\begin{align} \label{cont1} f_\tau(W_{2}^2(\mu_\tau, \nu_\tau)) -W_2^2(\mu,\nu) + \lambda_\omega \tau \omega(W_2^2(\mu,\nu_\tau)) \leq 2 \tau(E(\mu) - E(\mu_\tau)) + 2\lambda_\omega^2 c_r^2\tau^2\  .
\end{align}
Since $\tomega$ is a nondecreasing modulus of continuity for $\omega$ ,
\begin{align*}
&\omega(W_2^2(\mu,\nu_\tau)) - \omega(W_2^2(\mu,\nu)) \leq \tomega(|W_2^2(\mu,\nu_\tau) -W_2^2(\mu,\nu)|) \\
&\leq \tomega(|W_2(\mu,\nu_\tau) -W_2(\mu,\nu)||W_2(\mu,\nu_\tau) +W_2(\mu,\nu)|) \leq \tomega\left(W_2(\nu,\nu_\tau) (2W_2(\mu,\nu) +W_2(\nu,\nu_\tau) \right) \\
&\leq  \tomega(  W_2(\nu,\nu_\tau)(2R + 1)) \leq  \tomega \left( R^2 W_2(\nu,\nu_\tau) \right) \end{align*}
where, in the last step, we again use inequality (\ref{onestep1}) and the upper bounds on $\tau$.
Multiplying both sides of this inequality by $-\lambda_\omega\tau$, adding to (\ref{cont1}), and rearranging gives
\begin{align*}  &f_\tau(W_{2}^2(\mu_\tau, \nu_\tau)) \leq W_2^2(\mu,\nu)  -\lambda_\omega \tau \omega(W_2^2(\mu,\nu)) - \lambda_\omega \tau   \tomega \left( R^2  W_2(\nu,\nu_\tau) \right)  +  2 \tau(E(\mu) - E(\mu_\tau)) + 2\lambda_\omega^2 c_r^2\tau^2 .
\end{align*}

In order to apply  Lemma \ref{monotonicitylem} (\ref{ftau monotone}) a second time, we  obtain bound the right hand side by
\begin{align} \label{cont2} R^2 -\lambda_\omega \tomega(R^2) + 1 + 2< r  .
\end{align}
Thus, by Lemma \ref{monotonicitylem} (\ref{ftau monotone}) and (\ref{ftau break}),
\begin{align*} & f_\tau^{(2)}(W_{2}^2(\mu_\tau, \nu_\tau)) \\
&\quad \leq  f_\tau(W_2^2(\mu,\nu))   -\lambda_\omega \tau \omega(W_2^2(\mu,\nu)) - \lambda_\omega \tau  \tomega \left( R^2  W_2(\nu,\nu_\tau) \right)  +   2 \tau(E(\mu) - E(\mu_\tau)) +3\lambda_\omega^2 c_r^2\tau^2  , \\
&\quad  =  W_2^2(\mu,\nu) - \lambda_\omega \tau  \tomega \left( R^2  W_2(\nu,\nu_\tau) \right)  +    2 \tau(E(\mu) - E(\mu_\tau)) + 3\lambda_\omega^2 c_r^2\tau^2 .
\end{align*}
which gives the result for $\lambda_\omega \leq 0$.

\end{proof}

By iterating the contraction inequality, we are able to bound the Wasserstein distance between two discrete gradient flow sequences in terms of the distance between their initial data.
\begin{corollary}[$n$-step contraction inequality] \label{iteratedcocor}
Suppose $E$ satisfies assumption \ref{main assumptions} for $\lambda_\omega \leq 0$. Then for all $\mu,\nu \in D(E)$, and $t>0$, 
\[F_{2t}(W_2^2(\mu^n_{t/n}, \nu^n_{t/n} )) \leq W_2(\mu,\nu) +  o(1) \text{ as } n \to +\infty . \]
In particular, suppose that
\[n > t \max \{ 1, \tau_*^{-1}, (c_r \lambda_\omega)^2, R^2 , \} , \]
where $R\geq \max \{ W_2(\mu,\nu) +\sqrt{2(t+1)} \left(\sqrt{C(\mu,t)} + \sqrt{C(\nu,t)}\right), 3\}$ and $r :=4(t+1)( R^2+ |\lambda_\omega| \tomega(R^2))$.
Then,
\begin{align*}
&F_{2t}(W_2^2(\mu^n_{t/n}, \nu^n_{t/n}))  \leq W_2^2(\mu,\nu) + |\lambda_\omega| t \tomega \left(R^3 \sqrt{t/n} \right) + 2 Rt/n + 5 \lambda_\omega^2 c_r^2 t^2/n   + F_{2|\lambda_\omega|t} \left(\frac{2|\lambda_\omega| t \omega(R^2)}{n}  \right) .\nonumber
\end{align*}
\end{corollary}
\begin{proof}
Let $\tau:= t/n$. By inequalities (\ref{tau mu cons}) and (\ref{base case}), for all $j = 1, \dots, n$,
\begin{align} W_2(\nu^{j-1}_{\tau},\nu^{j}_{\tau}) \leq \sqrt{2\tau C(\mu,t)}  \leq R \sqrt{\tau}   \ \text{ and } \   W_2(\mu^{j}_{\tau}, \nu^{j}_{\tau}) \leq W_2(\mu, \mu^{j}_{\tau}) + W_2(\nu, \nu^{j}_{\tau}) + W_2(\mu, \nu)  \leq R.
 \label{iterated0}
\end{align}
We claim that it suffices to show, for all $j =1, \dots, n$,
\begin{align*}
&f^{(2j)}_{\tau}(W_2^2(\mu^{n}_\tau, \nu^{n}_\tau)) \leq W_2^2(\mu^{n-j}_\tau, \nu^{n-j}_\tau) + |\lambda_\omega| \tau j \tomega \left(R^3 \sqrt{\tau} \right) + 2\tau (E(\mu^{n-j}) - E(\mu^n_{\tau})) +5 \lambda_\omega^2 c_r^2 \tau^2 j .
\end{align*}
The result then follows by taking $j=n$ and applying inequality (\ref{tau mu def}) and Proposition \ref{odeproposition}.

We prove the claim by induction. To simplify notation, we suppress the subscripts on $\mu^j_\tau$ and $\nu^j_\tau$. 
The base case $j=1$ is a consequence of Theorem \ref{cothm} and inequalities (\ref{onestep1}) and (\ref{tau mu def}). Suppose that the result holds for $j-1$, i.e.,
\begin{align} \label{iterated1}
&f^{(2(j-1))}_{\tau}(W_2^2(\mu^{n}_\tau, \nu^{n}_\tau))\nonumber \\
&\quad \leq W_2^2(\mu^{n-j+1}_\tau, \nu^{n-j+1}_\tau) + |\lambda_\omega| \tau (j-1) \tomega \left(R^3  \sqrt{\tau} \right) + 2\tau(E(\mu^{n-j+1}) - E(\mu^n_{\tau})) +5 \lambda_\omega^2 c_r^2 \tau^2 (j-1) .
\end{align}
Note that for any $j =1, \dots, n$, the quantity on the right hand side is bounded by $r$.
Therefore, by Lemma \ref{monotonicitylem}, we may apply $f^{(2)}_\tau$ to both sides of (\ref{iterated1}) to obtain
\begin{align*}
f^{(2j)}_{\tau}(W_2^2(\mu^{n}_\tau, \nu^{n}_\tau)) 
& \leq f^{(2)}_\tau(W_2^2(\mu^{n-j+1}_\tau, \nu^{n-j+1}_\tau)) + |\lambda_\omega| \tau (j-1) \tomega \left(R^3 \sqrt{\tau} \right) + 2\tau(E(\mu^{n-j+1}) - E(\mu^n_{\tau})) \\
&\quad \quad  +5 \lambda_\omega^2 c_r^2 \tau^2 (j-1) + 2 \lambda_\omega^2 c_r^2 \tau^2 , \\
&\quad \leq W_2^2(\mu^{n-j}_\tau, \nu^{n-j}_\tau) + |\lambda_\omega| \tau j \tomega \left(R^3  \sqrt{\tau} \right) + 2\tau(E(\mu^{n-j}) - E(\mu^n_{\tau})) +5 \lambda_\omega^2 c_r^2 \tau^2 j ,
\end{align*}
where the second inequality follows from Theorem \ref{cothm}. This gives the result.
\end{proof}

\subsection{Convergence of discrete gradient flow} \label{convergence subsection}

We now use the estimates from the previous section to quantify the rate of convergence of the discrete gradient flow. Our argument relies strongly on convexity properties of the function $f_\tau(x)$. In general, our assumptions on the modulus of convexity $\omega(x)$ do not ensure $f_\tau(x)$ is convex. However, if an energy is $\omega$-convex with constant $\lambda_\omega$, Definition \ref{convexitydef} (\ref{omega convex def}) ensures that the energy is also $\tomega$-convex with constant $-\lambda_\omega^-$. Since $\tomega(x)$ is concave, this ensures that
\begin{align} \label{ftautildedef}
\tf_\tau(x) := \begin{cases} x - \lambda_\omega^- \tau \tomega(x) &\text{ if } x \geq 0, \\ 0 &\text{ if } x < 0.\end{cases}
 \end{align}
is always convex, which is sufficient for our estimates.

With these observations in hand, we now prove an asymmetric recursive inequality for the discrete gradient flow, analogous to a previous result by the author in the semiconvex case \cite[Theorem 2.4]{Craig}.
	
\begin{theorem}[asymmetric recursive inequality] \label{recineqthm} Suppose $E$ satisfies assumption \ref{main assumptions}. Then for all $T>0$ and $\mu \in D(E)$, there exists $\bar{\tau}>0$ and $\bar{C} \geq 0$, depending on $\tomega, \lambda_\omega^-,T,$ and $\tau_*$  so that for $0 \leq h \leq  \tau<   \bar{\tau}$ and $m, n \in \mathbb{N}$ with $mh, n \tau \leq T$,
\begin{align*}
 \tf^{(2m)}_h(W_2^2(\mu^n_\tau, \mu^m_h))  &\leq \frac{h}{\tau} \tf_h^{(2(m-1))}(W_2^2(\mu^{n-1}_\tau, \mu^{m-1}_h))  + \frac{\tau - h}{\tau} \tf_h^{(2(m-1))}(W_2^2(\mu^n,\mu^{m-1})) \\
 &\quad + \bar{C} \left[ h \tomega(\sqrt{\tau}) +  h^2 + \tomega(h^2) \right]+ 2h (E(\mu^{m-1}) - E(\mu^m)) .\end{align*}
\end{theorem}
\begin{remark}
If $|\partial E|(\mu^j) \leq C_\mu$ for all $1 \leq j \leq n$,  a similar argument gives an improved error estimate of $\bar{C} \left[ h \tomega( \tau) + h^2 + \tomega(h^2) \right]$, where $\bar{\tau}$ and $\bar{C}$ also depend on $C_\mu$. (This is analogous to \cite[Theorem 2.4]{Craig}.)
\end{remark}

\begin{proof}
To simplify notation, we suppress the subscripts on $\mu^n_\tau$ and $\mu^m_h$. First, we define $R$ to be sufficiently large and $\bar{\tau}$ sufficiently small for all previous results to apply. In particular, let
\begin{align}
 R&:= \max \left\{2 \sqrt{2 (T +1)C(\mu,T)}, 3 \right\} \nonumber \\
 r&:=4( R^2+ \lambda_\omega^- \tomega(R^2)), \label{rdef} \\
 \bar{\tau}&:= \min \left\{ 1,(c_r \lambda_\omega^-)^{-1},  R^{-2}/2, \tau_* \right\} \label{taubardef}
\end{align}
Define $\nu:={(\frac{\tau-h}{\tau}\bt_{\mu^{n-1}}^{\mu^n} + \frac{h}{\tau} \id) \# \mu^{n-1}}$, so by Lemma \ref{proxmapdifftimelem}, $\mu^n \in J_h \nu$. (We suppose that the optimal transport map $\bt_{\mu^{n-1}}^{\mu^n}$ exists for simplicity of notation---the analogous result holds in the more general case of optimal transport plans.)

Combining this with Theorem \ref{cothm},
\begin{align} \label{asym1}
\tf_h^{(2)}( W_2^2(\mu^n, \mu^m )) \leq W_2^2( \nu, \mu^{m-1} ) + \lambda_\omega^- h  \tomega\left(R^2 W_2(\nu,\nu_h) \right) + 2 h(E(\mu^{m-1}) - E(\mu^m)) + 3\lambda_\omega^2 c_r^2 h^2,
\end{align}
By Remark \ref{transport metric bounds W2} and equation (\ref{convexitytransportmetric}), we may bound the first term on the right hand side of (\ref{asym1}) by
\begin{align} \label{asyrecpf1}
W_2^2( \nu, \mu^{m-1} ) \leq W_{2, \bmu^{n-1}}^2( \nu, \mu^{m-1} ) \leq \frac{h}{\tau} W_{2}^2(\mu^{n-1},\mu^{m-1}) + \frac{\tau -
h}{\tau}W_{2, \bmu^{n-1}}^2(\mu^{m-1},\mu^n) .
\end{align}

The remaining estimates do not use the structure of the Wasserstein metric, so to ease notation, we abbreviate $W_{n,m} := W_2^2(\mu^n,\mu^m)$. Suppose the following claim holds.
\begin{align*}
&\text{\textbf{Claim:}} \\
&W_{2, \bmu^{n-1}}^2( \mu^{m-1},\mu^n) \leq \frac{\tau}{h} \left(W_{2}^2( \mu^n,\mu^{m-1}) - \tf_h(W_{2}^2( \mu^n,\mu^m)) \right)   +  W_2^2(\mu^{n-1},\mu^{m-1}) +  \lambda_\omega^- \tau  \tomega (W_2^2(\mu^n,\mu^m))  \\
&\quad + 2 \tau  \left(E(\mu^{m-1}) - E(\mu^m) \right) + \lambda_\omega^- \tau \left (2 \tomega( R^2 W_2(\mu^{n-1},\mu^n)) +\tomega( R^2 W_2(\mu^{m-1},\mu^m) \right)+3(\lambda_\omega^-)^2 c_r^2 h \tau
\end{align*}
Substituting this into inequality (\ref{asyrecpf1}) and then both into inequality (\ref{asym1}),
\begin{align} \label{asyrecpf25}
&\tf^{(2)}_h(W_{n,m}) \nonumber \\
&\quad \leq  \frac{h}{\tau} W_{n-1,m-1}+ \frac{\tau - h}{h}\left(W_{n,m-1} - \tf_h(W_{n,m}) \right)  +   \frac{\tau-h}{\tau} \left(W_{n-1,m-1} + \lambda_\omega^- \tau\tomega(W_{n,m})\right) + D_{n,m} ,
\end{align}
where the error term $D_{n,m}$ is given by
\begin{align*}
D_{n,m} &=  \lambda_\omega^- h  \tomega\left(R^2 W_2(\nu,\nu_h) \right) + 2  \tau \left(E(\mu^{m-1}) - E(\mu^m) \right) + 3(\lambda_\omega^-)^2 c_r^2 \tau h  \\
&\quad +  \lambda_\omega^- (\tau-h) \left(  2   \tomega(R^2 W_2(\mu^{n-1},\mu^n)) + \tomega(R^2 W_2(\mu^{m-1},\mu^m)) \right)  .
\end{align*}
Adding $\frac{\tau-h}{h} \tf_h^{(2)}(W_{n,m})$ to both sides of (\ref{asyrecpf25}) and using $\tf_h^{(2)}(W_{n,m}) - \tf_h(W_{n,m}) = -\lambda_\omega^- h \tomega (\tf_h(W_{n,m}))$,
\begin{align}\label{asyrecpf3}
\frac{\tau}{h} \tf^{(2)}_h(W_{n,m}) &\leq W_{n-1,m-1} + \frac{\tau - h}{h}W_{n,m-1} +   \lambda^-_\omega  (\tau- h)  \left[ \tomega(W_{n,m}) - \tomega (\tf_h(W_{n,m}) ) \right] +D_{n,m} .
\end{align}
To control the second to last term in this inequality, note that since $\tomega$ is subadditive and nondecreasing, 
\begin{align*}
 \tomega(W_{n,m}) - \tomega (\tf_h(W_{n,m}) ) \leq \tomega(|W_{n,m} - \tf_h(W_{n,m})|)  \leq \tomega(\lambda_\omega^- h \tomega(R^2)) \leq \left(\lambda_\omega^- \tomega(R^2)+1 \right) \tomega(h) .
\end{align*}
Substituting this into (\ref{asyrecpf3}), multiplying by $h/\tau$ and using $h/\tau \leq 1$,
\begin{align} \label{asyrecpf325}
 \tf^{(2)}_h(W_{n,m}) &\leq \frac{h}{\tau} W_{n-1,m-1}  + \frac{\tau - h}{\tau}W_{n,m-1}   +  \left(\lambda_\omega^- \tomega(R^2)+1 \right)  \lambda^-_\omega h \tomega(h)+ \frac{h}{\tau} D_{n,m} .
 \end{align}
Since  $ \lambda^-_\omega h \tomega(h) \leq \lambda^-_\omega c_r h^{3/2} \leq 1$,
the right hand side of (\ref{asyrecpf325}) is bounded by 
\begin{align} \label{asyrecpf35} \frac{h}{\tau} R^2 + \frac{\tau-h}{\tau} R^2+ ( \lambda_\omega^- \tomega(R^2)+1) + \frac{h}{\tau} \left( 3 \lambda_\omega^- \tomega(R^2) + 1 + 3 \right) \leq r . 
\end{align}
Thus, we may apply $\tf_h^{(2(m-1))}$ to both sides, using Lemma \ref{monotonicitylem} and the convexity of $f_h$,
\begin{align} \label{asyrecpf4}
 \tf^{(2m)}_h(W_{n,m})  &\leq \frac{h}{\tau} \tf_h^{(2(m-1))}(W_{n-1,m-1})  + \frac{\tau - h}{\tau} \tf_h^{(2(m-1))}(W_{n,m-1}) \nonumber \\
 &\quad+ \left(\lambda_\omega^- \tomega(R^2)+1 \right) \lambda_\omega^- h \tomega(h) + \frac{h}{\tau} D_{n,m}
+ 2(m-1) \lambda_\omega^- h \tomega((\lambda_\omega^-)^2 c_{r}^2 h^2) .
\end{align}
Finally, we may bound the second line of inequality (\ref{asyrecpf4}) by
\begin{align} \label{asyrecpf5}
\bar{C} \left[ h \tomega(\sqrt{\tau}) +  h^2 + \tomega(h^2) \right] + 2h (E(\mu^{m-1}) - E(\mu^m)),
\end{align}
with 
\begin{align}\label{bar C1 def}
\bar{C} := \lambda_\omega^- \max\{  \lambda_\omega^- \tomega(R^2)+ 4R^2 , 3 \lambda_\omega^- c_r^2, 2T((\lambda_\omega^-)^2 c_r^2 +1)\} +R.
\end{align}
(While the addition of $R$ is not needed to bound (\ref{asyrecpf4}), in what follows, it is convenient that $\bar{C} \geq R$.)

It remains to prove the claim. By Proposition \ref{discreteEVI},
\begin{align}
\tf_h(W_{2,\bmu^{m-1}}^2(\mu^m, \mu^n))-W_2^2(\mu^{m-1},\mu^n)  &\leq 2h \left( E(\mu^n)- E(\mu^m) \right)   , \label{transpbd1} \\
\tf_\tau( W_{2,\bmu^{n-1}}^2(\mu^n, \mu^{m-1}))-W_2^2(\mu^{n-1},\mu^{m-1})  &\leq 2 \tau \left( E(\mu^{m-1})- E(\mu^n) \right) . \label{asympf2}
\end{align} 
Multiplying (\ref{transpbd1}) by $\tau$, (\ref{asympf2}) by $h$, and adding them together gives 
\begin{align}\label{mixedfimestepsvarineq} 
&\tau \tf_h( W_{2,\bmu^{m-1}}^2(\mu^m, \mu^n)) + h \tf_\tau(W_{2, \bmu^{n-1}}^2(\mu^n, \mu^{m-1}))\nonumber \\
&\quad \leq \tau W_2^2(\mu^{m-1},\mu^n) + h W_2^2(\mu^{n-1},\mu^{m-1}) + 2 \tau h \left(E(\mu^{m-1}) - E(\mu^m) \right) 
\end{align}
Since $\tomega$ is a nondecreasing, concave modulus of continuity for $\tomega$, the reverse triangle inequality, Remark \ref{transport metric bounds W2},  and $h \leq \tau$ give
\begin{align*}
&\tomega(W_{2, \bmu^{n-1}}^2(\mu^n,\mu^{m-1})) - \tomega (W_2^2(\mu^n,\mu^m))  \\
&\leq\tomega(W_{2, \bmu^{n-1}}^2(\mu^n,\mu^{m-1})) - \tomega(W_{2,\bmu^{n-1}}^2(\mu^{n-1},\mu^{m-1}))\\
&\quad  
+\tomega(W_2^2(\mu^{n-1},\mu^{m-1})) - \tomega(W_2^2(\mu^{n}, \mu^{m-1})) + \tomega(W_2^2(\mu^{n}, \mu^{m-1}))- \tomega (W_2^2(\mu^n,\mu^m)) \\
&\leq \tomega\left(W_2(\mu^{n-1},\mu^n) [W_{2,\bmu^{n-1}}(\mu^n,\mu^{m-1}) + W_2(\mu^{n-1},\mu^{m-1})]\right) \\
&\quad 
+\tomega\left(W_2(\mu^{n-1},\mu^n) [W_{2}(\mu^{n-1},\mu^{m-1}) + W_2(\mu^{n},\mu^{m-1})]\right)  +\tomega\left(W_2(\mu^{m-1},\mu^m) [W_{2}(\mu^{n},\mu^{m-1}) + W_2(\mu^{n},\mu^{m})]\right) \\
&\leq 2\tomega(2R W_2(\mu^{n-1},\mu^n)) +\tomega(2R W_2(\mu^{m-1},\mu^m))
\end{align*}
Multiplying both sides of this inequality by $\lambda_\omega^{-} \tau h$, and adding to (\ref{mixedfimestepsvarineq}) gives
\begin{align} \label{transpbd3}
&\tau \tf_h(W_{2,\bmu^{m-1}}^2(\mu^m, \mu^n)) + h  W_{2, \bmu^{n-1}}^2(\mu^n, \mu^{m-1}) \nonumber \\
& \leq \tau W_2^2(\mu^{m-1},\mu^n) +h  W_2^2(\mu^{n-1},\mu^{m-1}) + \lambda_\omega^- \tau h \tomega (W_2^2(\mu^n,\mu^m)) + 2 \tau h \left(E(\mu^{m-1}) - E(\mu^m) \right)\nonumber \\
&\quad +  \lambda_\omega^- \tau h\left (2 \tomega(R^2W_2(\mu^{n-1},\mu^n)) +\tomega(R^2W_2(\mu^{m-1},\mu^m) \right) .
\end{align}
By Remark \ref{transport metric bounds W2}, $W_{2}^2(\mu^m, \mu^n) \leq W_{2,\bmu^{m-1}}^2(\mu^m, \mu^n)$. Thus, applying Lemma \ref{monotonicitylem},
\[ \tau \tf_h(W_{2}^2(\mu^m, \mu^n))  \leq \tau \tf_h(W_{2,\bmu^{m-1}}^2(\mu^m, \mu^n)) + (\lambda_\omega^-)^2 c_R^2 h^2 \tau  \leq \tau \tf_h(W_{2,\bmu^{m-1}}^2(\mu^m, \mu^n)) +3 (\lambda_\omega^-)^2 c_r^2 h^2 \tau. \]
Substituting this into (\ref{transpbd3}), rearranging, and dividing by $h$ proves the claim.

\end{proof}

Next, we use this asymmetric recursive inequality, arguing by induction, to prove a quantitative bound on the Wasserstein distance between two discrete gradient flow sequences with the same initial conditions and different  time steps $\tau$ and $h$. This is analogous to a previous result by the author in the semiconvex case \cite[Theorem 2.6]{Craig}.
 
\begin{theorem}[distance between discrete gradient flows] \label{W2RasBound} 
Suppose $E$ satisfies assumption \ref{main assumptions}. Then for all $T>0$ and $\mu \in D(E)$, there exists $\bar{\tau}>0$ and $\bar{C} \geq 0$, depending on $\tomega, \lambda_\omega^-, T,$ and $\tau_*$, so that for $0 \leq h \leq  \tau<   \bar{\tau}$ and $m,n \in \mathbb{N}$ with $mh, n \tau \leq T$,
\begin{align*}
&\tf^{(2m)}_h (W_2^2(\mu^n_\tau, \mu^m_h))  \\
 &\quad \leq \bar{C}\left[\sqrt{(n\tau - mh)^2+\tau^2n}  + h m\tomega(\sqrt{\tau}) +h^2m + \tomega(h^2)m \right] +2h (E(\mu) -E(\mu_h^m)) .
\end{align*}
\end{theorem}
\begin{remark}
If $|\partial E|(\mu^j) \leq C_\mu$ for all $1 \leq j \leq n$,  a similar argument gives an improved estimate of $\bar{C} \left[(n\tau - mh)^2+  \tau^2 n  + h m\tomega(\tau) +h^2m + \tomega(h^2)m \right]$,  where $\bar{\tau}$ and $\bar{C}$ also depend on $C_\mu$. (This is analogous to \cite[Theorem 2.6]{Craig}.)
\end{remark}

\begin{proof}
Let $\bar{\tau}$ be as in equation (\ref{taubardef}) and $\bar{C}$ be as in (\ref{bar C1 def}).
We proceed by induction. The base case, when $n =0$ or $m =0$, follows from the fact that $\tf_h(x) \leq x$ and inequality (\ref{base case}).
We assume the inequality holds for
$(n-1,m)$ and $(n,m)$ and show that this implies it holds for $(n, m+1)$. 
By Theorem \ref{recineqthm},
\begin{align*}
 \tf^{(2(m+1))}_h(W_2^2(\mu^n_\tau, \mu^{m+1}_h)) &\leq \frac{h}{\tau} \tf_h^{(2m)}(W_2^2(\mu^{n-1}_\tau, \mu^{m}_h))  + \frac{\tau - h}{\tau} \tf_h^{(2m)}(W_2^2(\mu^n,\mu^{m})) \\
&\quad + \bar{C} \left[ h \tomega(\sqrt{\tau}) +  h^2 + \tomega(h^2) \right]+ 2h (E(\mu^{m}) - E(\mu^{m+1})) .
\end{align*}
Applying the inductive hypothesis,
\begin{align} \label{pre exp form2}
\tf^{(2(m+1))}_h (W_2^2(\mu^n_\tau, \mu^{m+1}_h))  &\leq   \frac{h}{\tau} \bar{C}  \sqrt{((n-1)\tau - mh)^2 +\tau^2  (n-1)}  + \frac{\tau-h}{\tau} \bar{C}  \sqrt{ (n\tau - mh)^2+ \tau^2  n}  \nonumber \\ 
&\quad + \bar{C} \left[ h m \tomega(\sqrt{\tau}) + h^2m+ \tomega(h^2)m \right]+ 2h (E(\mu) - E(\mu^m)) \nonumber \\
&\quad + \bar{C} \left[ h \tomega( \sqrt{ \tau})+ h^2 + \tomega(h^2) \right] + 2h (E(\mu^{m}) - E(\mu^{m+1})).
\end{align}
To control the first term, note that 
\begin{align} \label{pre exp form1}
& ((n-1)\tau - mh)^2 + \tau^2  (n-1) =  (n\tau - mh)^2 + \tau^2 n  - 2\tau(n \tau - mh)   .
\end{align}
Combining this identity with the concavity of $\sqrt{\cdot}$, we may bound the first line in inequality (\ref{pre exp form2}) by
\[ \bar{C}\sqrt{ (n \tau - mh)^2 +\tau^2 n - 2h(n \tau -mh) } \leq \bar{C}\sqrt{ (n \tau - (m+1)h)^2 +\tau^2 n } .\]
Therefore,
\begin{align*}
&\tf^{(2(m+1))}_h (W_2^2(\mu^n_\tau, \mu^{m+1}_h))  \\
&\leq  \bar{C} \left[ \sqrt{ (n \tau - (m+1)h)^2 +\tau^2 n }  + h (m+1) \tomega(\sqrt{\tau}) + h^2(m+1)+ \tomega(h^2)(m+1) \right] + 2h (E(\mu) - E(\mu^{m+1})),
\end{align*}
which gives the result.
\end{proof}

Taking $\tau = t/n$ and $h = t/m$ in the previous theorem, we conclude that the discrete gradient flow sequence $\mu^n_{t/n}$ is Cauchy and quantify its rate of convergence.

\begin{theorem}[convergence of discrete gradient flow] \label{expform}
Suppose $E$ satisfies assumption \ref{main assumptions}. Then, for any $t\geq 0$ and $\mu \in D(E)$, $\mu^n_{t/n}$ converges to a limit $\mu(t)$ as $n \to +\infty$. Furthermore, for any $T>0$ there exist positive constants $\bar{N}$ and $\bar{C}$, depending on $\tomega, \lambda_\omega^-, T,$ and $\tau_*$, so that
\begin{align} \label{exp form eqn2}
 \tF_{2t} \left(W_2^2(\mu^n_{t/n}, \mu(t)) \right) \leq \bar{C} \left[t/\sqrt{n}+t \tomega(\sqrt{t/n})\right] \text{ for all } 0 \leq t \leq T \text{ and } n > \bar{N} .
 \end{align}
\end{theorem}

\begin{remark} \label{improved rate}
If $|\partial E|(\mu^j) \leq C_\mu$ for all $1 \leq j \leq n$,  a similar argument gives an improved estimate of $\bar{C}\left[ t^2/n + t\tomega(t/n)  \right]$, where $\bar{N}$ and $\bar{C}$ also depend on $C_\mu$. (This is analogous to \cite[Theorem 2.7]{Craig}.)
\end{remark}

\begin{example}[rate of convergence] \label{rate convergence example}
For the three moduli of convexity from example \ref{omega examples}, we obtain the following rates of convergence:
\begin{enumerate}[(i)]
\item Lipschitz and polynomial modulus of convexity:  $W_2(\mu^n_{t/n}, \mu(t)) =O(n^{-1/4})$; \label{lip rate}
 \item  log-Lipschitz modulus of convexity: \label{log lip rate} $W_2(\mu^n_{t/n},\mu(t)) = \mathcal{O} \left( \left[  n^{-1/2} \log(n) \right]^{1/2e^{2\lambda_\omega^- t}}  \right)$.
\end{enumerate}
\end{example}

\begin{proof}
Let $\bar{\tau}$ be as in equation (\ref{taubardef}) and $\bar{C}$ be as in (\ref{bar C1 def}).
By Theorem \ref{W2RasBound}, for $m \geq n > t/\bar{\tau}$, and inequality (\ref{tau mu def}),
\begin{align*} 
&\tf^{(2m)}_h (W_2^2(\mu^n_\tau, \mu^m_h))  \leq \bar{C}  \left[t/\sqrt{n}  + t \tomega(\sqrt{t/n}) +t^2/m + \tomega(t^2/m^2)m \right]+2(t/m) C(\mu,t) .
\end{align*}
Therefore, for $R := 2 \sqrt{2t C(\mu,t)} \geq W_2(\mu^n_{t/n}, \mu^m_{t/m})$,
Proposition \ref{odeproposition} ensures
\begin{align} &\tF_{2t}(W_2^2(\mu^n_{t/n}, \mu^m_{t/m})) \nonumber \\
&\quad\leq \bar{C} \left[t/\sqrt{n}  + t \tomega(\sqrt{t/n}) +t^2/m + \tomega(t^2/m^2)m \right]
 +R^2(t/m) +  \tF_{\lambda_\omega^- 2t}\left( \frac{\lambda_\omega^- 2t  \tomega(R^2)}{m}  \right) . \label{exp form1}
 \end{align}
Consequently, the sequence $\mu^n_{t/n}$ is Cauchy. Since the Wasserstein metric is complete \cite[Proposition 7.1.5]{AGS}, the limit $\mu(t)$ exists. Sending $m \to +\infty$ in the above inequality and using that $\tomega(x) = o(\sqrt{x})$ as $x \to 0$ gives the result.
\end{proof}

\subsection{Properties of gradient flow} \label{properties section}
In this section, we conclude our proof of well-posedness of the gradient flow for $\omega$-convex functions by showing that the limit of any discrete gradient flow sequence is the unique solution of the gradient flow. We also use our $n$-step contraction inequality for the discrete gradient flow (Theorem \ref{iteratedcocor}) to prove a  contraction inequality for the gradient flow.

\begin{theorem} \label{propertiesoflimit}
Suppose $E$ satisfies assumption \ref{main assumptions}. Then for all $t \in [0, +\infty)$, the function 
\[ S(t): D(E) \to D(E): \mu \mapsto \mu(t) \]
 is a strongly continuous semigroup. In particular, for any $\mu,\nu \in D(E)$,
\begin{enumerate}[(i)]
\item $\lim_{t \to 0} S(t) \mu = S(0) \mu = \mu$; \label{ctyitem}
\item$S(t + s) = S(t) S(s) \mu$ for $t, s \geq 0$; \label{semigroupitem}
\item $F_{2t}(W_2^2(S(t) \mu, S(t) \nu)) \leq W_2^2(\mu,\nu)$;
\label{contractionitem}
\item \label{grad flow item} $S(t)\mu$ is the unique gradient flow of $E$, in the sense of Definition \ref{gradflowdef}.
\end{enumerate}
\end{theorem}
\begin{remark}[curve of maximal slope]
In addition to the fact that $S(t)\mu$ is the gradient flow of $E$, in the sense of Definition \ref{gradflowdef}, one can also show that $S(t) \mu$ is a curve of maximal slope of $E$ with respect to $|\partial E|$, in the sense of \cite[Definition 1.3.2]{AGS}. This is a consequence of \cite[Theorem 2.3.3]{AGS}, since $|\partial E|$ is both lower semicontinuous (Proposition \ref{slope lsc}) and is a strong upper gradient for $E$ (equation (\ref{HWI slope formula}) and \cite[Corollary 2.4.10]{AGS}).
 \end{remark}

\begin{example}[rate of contraction]	\label{rate of contraction example}
For the three moduli of convexity from Example \ref{omega examples}, we obtain the following rates of contraction in Theorem \ref{propertiesoflimit} (\ref{contractionitem}):
\begin{enumerate}[(i)]
\item Lipschitz modulus of convexity:  $W_2(S(t)\mu,S(t)\nu) \leq e^{-\lambda_\omega t} W_2(\mu,\nu)$ for all $t \geq 0$; \label{lip contr}
 \item polynomial modulus of convexity: $W_2(S(t)\mu,S(t)\nu) \leq W_2(\mu,\nu)(1+2 \lambda_\omega pt W_2^{2p}(\mu,\nu))^{-1/2p}$ \\
 for $0 \leq W_2(\mu,\nu) \leq 1$ and $t < (W_2(\mu,\nu)^{-2p}-1) /2\lambda_\omega^- p$; \label{poly contr}
\item log-Lipschitz modulus of convexity: \label{log lip contr} $W_2(S(t) \mu, S(t)\nu) \leq W_2(\mu,\nu)^{1/e^{-2 \lambda_\omega t}}$ \\
for $0 \leq W_2(\mu,\nu) \leq e^{-1-\sqrt{2}}$ and $t < \log \left( \log(W^2_2(\mu,\nu))/(-1-\sqrt{2}) \right)/2\lambda_\omega^-$.
\end{enumerate}
\end{example}

\begin{proof}
First, note that by \cite[Corollary 3.3.4]{AGS}, $S(t)\mu$ is locally absolutely continuous in $t$. 
When $\lambda_\omega \leq 0$, (\ref{contractionitem}) is an immediate consequence of Corollary \ref{iteratedcocor}. If $E$ is $\omega$-convex with constant $\lambda_\omega >0$, it is also $\tomega$-convex with constant $\lambda_{\tomega}=0$, so (\ref{contractionitem}) provides a coarse stability estimate for all $\lambda_\omega \in \R$. We will show the improved estimate for $\lambda_\omega >0$ at the end of this proof.

Now, we turn to (\ref{semigroupitem}). First, note that it suffices to show $S(t)^m \mu = S(mt) \mu$ for fixed $m \in \mathbb{N}$. In particular, if this estimate holds, then for any $l,k,r,s \in \mathbb{N}$,
\begin{align*}
 S\left(\frac{l}{k} + \frac{r}{s} \right) \mu &= S\left(\frac{ls + rk}{ks} \right) \mu = \left[S \left(\frac{1}{ks} \right) \right]^{ls + rk}  \mu  = \left[ S \left(\frac{1}{ks} \right) \right]^{ls} \left[ S \left(\frac{1}{ks} \right) \right]^{rk} \mu  = S\left(\frac{l}{k} \right) S\left( \frac{r}{s} \right) \mu   .
 \end{align*}
Using that $S(t) \mu$ is continuous in $t$, we obtain $S(t+s) \mu= S(t) S(s) \mu$ for all $t, s \geq 0$.

To show $S(t)^m \mu = S(mt) \mu$, we proceed by induction. The result holds for $m =1$, so suppose that $S(t)^{m-1} \mu = S((m-1)t) \mu$. With slight abuse of notation, we briefly use $J_\tau \mu$ to denote any element in this set. By the triangle inequality,
\begin{align} \label{semigroupeqn1}
W_2(S(mt) \mu, S(t)^m \mu) &\leq W_2(S(mt) \mu, (J^n_{t/n})^m \mu) + W_2( (J^n_{t/n})^{m} \mu, J^n_{t/n} S(t)^{m-1} \mu) \nonumber \\
&\quad + W_2( J^n_{t/n} S(t)^{m-1} \mu, S(t)^{m} \mu)
\end{align}
For the first and third terms, Theorem \ref{expform} ensures
\begin{align*}
W_2(S(mt) \mu, (J^n_{t/n})^m \mu) &=W_2( S(mt) \mu,J^{nm}_{tm/nm} \mu) \xrightarrow{n \to \infty} 0 , \\
W_2( J^n_{t/n} S(t)^{m-1} \mu, S(t)^{m} \mu) &= W_2( J^n_{t/n} S(t)^{m-1} \mu, S(t) S(t)^{m-1}  \mu) \xrightarrow{n\to \infty} 0 .
\end{align*}
We now consider the second term in equation (\ref{semigroupeqn1}). 
By the inductive hypothesis and Theorem \ref{expform},
\begin{align*}
 W_2((J^n_{t/n})^{m-1} \mu, S(t)^{m-1} \mu) = W_2(J^{n(m-1)}_{t(m-1)/n(m-1)} \mu, S((m-1)t) \mu) \xrightarrow{n \to \infty}0.
 \end{align*}
Therefore, by Corollary \ref{iteratedcocor},
 \begin{align*}
&\tF_{2t}(W_2^2(J^n_{t/n} (J^n_{t/n})^{m-1} \mu, J^n_{t/n} S(t)^{m-1} \mu)) \\
& \leq W^2_2((J^n_{t/n})^{m-1} \mu,S(t)^{m-1} \mu) + \lambda_\omega^- t \tomega \left(R^3 \sqrt{t/n} \right) + 2Rt/n + 5 \lambda_\omega^2 c_r^2 t^2/n + F_{2\lambda_\omega^- t} \left(\frac{2\lambda_\omega^- t \omega(R^2)}{n}  \right)  ,
\end{align*}
which goes to zero as $n \to +\infty$.
Consequently, we conclude that  $S(t)^m \mu = S(mt) \mu$ for all $m \in \mathbb{N}$.

We now prove part (\ref{grad flow item}). By part (\ref{semigroupitem}), it suffices to show that for all $\nu \in D(E)$ and $t>0$,
\begin{align} \label{grad flow item1}
W_2^2(S(t) \mu, \nu) -W_2^2(\mu,\nu) \leq \int_0^t 2(E(\nu) - E( S(s)\mu))-\lambda_\omega \omega( W_2^2(S(s) \mu, \nu)) ds .
\end{align}
Note that for all $i =1, \dots, n$ and $n$ sufficiently large,
\begin{align} \label{fatou bound} W^2_2(\mu^i_{t/n}, \nu) \leq \left(W_2(\mu,\nu) + \sqrt{2 t(E(\mu) - E(S(t)\mu) + 1)} \right)^2 =:R 
\end{align}
Consequently, by Remark \ref{transport metric bounds W2}, Lemma \ref{monotonicitylem} (\ref{ftau monotone}), and Proposition \ref{discreteEVI}, for $n> t c_R |\lambda_\omega|$,
\begin{align*}
W_2^2(\mu^i_{t/n}, \nu) - W_2^2(\mu^{i-1}_{t/n}, \nu) \leq 2 (t/n) (E(\nu) - E(\mu^i_{t/n})) - \lambda_\omega (t/n)\omega(W_2^2(\mu^i_{t/n},\nu)) + \lambda_\omega^2 c_R^2 ({t/n})^2 .
\end{align*}
Summing both sides from $i=1,\dots, n$,
\begin{align} \label{prefatou}
W_2^2(\mu^n_{t/n},\nu) - W_2^2(\mu,\nu) &\leq (t/n) \sum_{i=1}^n \left[2(E(\nu) - E(\mu^i_{t/n})) - \lambda_\omega \omega(W_2^2(\mu^i_{t/n},\nu)) +  \lambda_\omega^2 c_R^2 t/n \right] , \nonumber \\
&=  \int_0^{t} \psi_n(s) ds + \lambda_\omega^2 c_R^2 t^2/n .
\end{align}
where $\psi_n(s) := 2(E(\nu) - E(\mu^i_{t/n})) - \lambda_\omega \omega(W_2^2(\mu^i_{t/n},\nu)) $ for $s \in ((i-1){t/n}, i{t/n}]$, $1 \leq i \leq n$. By inequality (\ref{fatou bound}) and the fact that $E$ decreases along the discrete gradient flow, $\psi_n(s)$ is uniformly bounded above. As before, $\mu^i_{t/n}\xrightarrow{n \to +\infty} S(s) \mu$, hence by the lower semicontinuity of $E$,
\begin{align} \label{fatou} \limsup_{n \to \infty} \psi_n(s) \leq 2(E(\nu)-E(S(s)\mu)) - \lambda_\omega \omega(W_2^2(S(s)\mu,\nu)) .
\end{align} 
Therefore, sending $n \to +\infty$ in (\ref{prefatou}) gives inequality (\ref{grad flow item1}) by Fatou's lemma.

Finally, we prove (\ref{contractionitem}) when $\lambda_\omega \geq 0$. By part (\ref{grad flow item}),
\[ \frac{d}{dt} W_2^2(S(t) \mu, S(t) \nu) \leq -2 \lambda_\omega \omega \left(W_2^2(S(t)\mu, S(t)\nu) \right) . \]
The result then follows by Bihari's inequality \cite{Bihari}.

\end{proof}

\section{Application:  constrained nonlocal interaction and nonconvex drift diffusion} \label{applicationssection}
In this final section, we apply our previous results to gradient flows of \emph{constrained nonlocal interaction} energies. Our interest in such energies is inspired by (unconstrained) nonlocal interactions equations of the form
\[ \partial_t \rho = \grad \cdot ((\grad W *\rho) \rho) + D \Delta \rho^m , \quad W: \Rd \to \R,  \ D \geq 0, \ m \geq 1 . \]
These types of partial differential equations arise in a range of applications, including models of granular media (c.f. \cite{BenedettoCagliotiCarrilloPulvirenti, CarrilloMcCannVillani}) and biological swarming (c.f. \cite{TopazBertozziLewis, Blanchet, KellerSegel, BurgerFetecauHuang}), and can be thought of as the mean-field limit of a system of a large number of interacting agents \cite{CarrilloChoiHauray}. A guiding principle in the choice of interaction potential $W$ and the diffusion parameters $D, m$ is a desire for short-range repulsion, which prevents  agents from colliding, and long-range attraction, which confines agents to a bounded set, thereby allowing coherent structures to develop in the long-time limit.

Motivated by these (unconstrained) nonlocal interaction models, the author, Kim, and Yao consider a model of constrained nonlocal interaction, in which the interaction potential $W$ is purely attractive and the role of repulsion is played by a ``hard height constraint'' on the density \cite{CraigKimYao}. Informally, the model is given by
\[\begin{cases} \partial_t \rho = \grad \cdot (\grad (W*\rho) \rho) \text{ if }\rho < 1, \\
 \rho \leq 1 \text{ always.} \end{cases} \]
 Rigorously, the model is posed as the Wasserstein gradient flow of the constrained nonlocal interaction energy 
 \begin{align} \label{E infty first}
  \mathcal{W}_\infty(\mu) := \begin{cases} \frac{1}{2} \int \mu(x) W*\mu(x) dx &\text{if $\| \mu\|_\infty \leq 1$}, \\ +\infty & \text{otherwise.}  \end{cases}  \end{align}
 
 Similar models have been studied in Wasserstein gradient flow context by Maury, Roudneff-Chupin, Samtambrogio, and Venel \cite{MRS, MRSV} and M\'esz\'aros and Santabrogio \cite{MeszarosSantambrogio} when the nonlocal interaction term $W*\mu(x)$ is replaced by a local potential $V(x)$. In the case that $V \in C^2(\Rd)$ satisfies $D^2 V(x) \geq \lambda \rm{Id}_{d \times d}$ for all $x \in \Rd$---i.e. is semiconvex as a function on Euclidean space--- Alexander, Kim, and Yao showed that the corresponding Wasserstein gradient flows could be approximated by viscosity solutions of related partial differential equations. Specifically, they used Ambrosio, Gigli, and Savar\'e's quantitative estimates on the convergence of the discrete gradient flow for semiconvex energies to show that solutions of the porous medium equation with drift
   \[ \partial_t \rho = \grad \cdot (\grad V \rho) +  \Delta \rho^m , \quad m \geq 1 .    \]
converge as $m \to \infty$ to the gradient flow of the constrained potential energy
  \[ \mathcal{V}_\infty(\mu) := \begin{cases} \frac{1}{2} \int \mu(x) V(x) dx &\text{if $\| \mu\|_\infty \leq 1$}, \\ +\infty & \text{otherwise.}  \end{cases}     \]
  Using this connection between viscosity solutions and gradient flows, Alexander, Kim, and Yao were then able to characterize solutions of the gradient flow of $E_\infty$ with characteristic function initial data in terms of a Hele-Shaw type free boundary problem, thus providing a concrete PDE characterization of what had previously been  an abstract Wasserstein gradient flow.
  
A key challenge in adapting these results on constrained potential energies to constrained nonlocal interaction energies was the absence of quantitative estimates on the rate of convergence of the discrete gradient flow outside of the semiconvex case.
In particular, the author, Kim, and Yao were interested in the case when $W(x)$ was given by the Newtonian potential,
\begin{align} \label{newtonian potential} 
W(x) = \begin{cases}\frac{1}{2\pi}\log|x| & \text{ for }d=2,\\ \frac{1}{d(2-d)\alpha_d} |x|^{2-d} & \text{ for } d\geq 3,
\end{cases}
\end{align}
for which the corresponding constrained nonlocal interaction energy (\ref{E infty first}) fell outside the scope of the existing semiconvex theory.

 This obstacle inspired the present work. In spite of its lack of semiconvexity,  the constrained nonlocal interaction energy, with $W(x)$ given by (\ref{newtonian potential}), is $\omega$-convex for a log-Lipschitz modulus of convexity, and by Theorems \ref{expform} and \ref{propertiesoflimit} and Examples \ref{Newt att ex} (\ref{cong agg}) and \ref{rate convergence example}, the discrete gradient flow converges to the unique solution of the gradient flow with rate $\left(  n^{-1/2} \log(n) \right)^C$, where $C>0$ depends on the dimension and the time interval on which solutions are considered. With these results in hand, the author, Kim, and Yao were then able to give a PDE characterization of the congested aggregation model and study its long-time behavior.
 
 In this section, we show that our previous results on well-posedness of Wasserstein gradient flow and convergence of the discrete gradient flow apply to a range of constrained nonlocal interaction energies and nonconvex drift diffusion energies of the following forms:
\begin{align} \label{Wp def}
\mathcal{W}_p(\mu) &:= \begin{cases}
\frac{1}{2} \int \mu(x) W*\mu(x) dx &\text{if $\| \mu\|_p \leq C_p$}, \\ +\infty & \text{otherwise,}  \end{cases}       \\
\mathcal{V}_m(\mu) &:= \begin{cases}  \int_\Rd \mu(x)V(x) dx + \frac{1}{m-1} \int \mu^m(x)dx  &\text{ if } \mu \ll \mathcal{L}^d, \\
+ \infty &\text{ otherwise.}
\end{cases} \label{V def}
\end{align}
We allow $m = +\infty$ in the drift diffusion energy $\mathcal{V}_m$ by replacing the second term by 
\[ I_\infty(\mu) := \begin{cases} 0 &\text{ if } \|\mu\|_\infty \leq 1, \\ +\infty &\text{ otherwise.}\end{cases} \]

In particular, we consider $W, p, V,$ and $m$ that satisfy the following properties:
\begin{assumption}[constrained nonlocal interaction energy] \label{interaction assumption}
There exist $C, C'>0$ so that
\begin{enumerate}[(i)]
\item  $W^-*\mu(x) \leq C'$ for all $\mu \in D(\mathcal{W}_p)$. \label{bddbelow ass} 
\item For any $\mu, \nu \in \P_2(\Rd)$ with $\| \mu \|_p  \leq C_p$, $\|\grad W*\mu\|_{L^2(\nu)} \leq C'$. \label{subdiffbound}
\item For any $\mu$ with $\|\mu\|_p \leq C_p$, $W*\mu(x)$ is continuously differentiable, and there exists a continuous, nondecreasing, concave function $\psi: [0, +\infty) \to [0, +\infty)$ satisfying $\psi(0)=0$, $\psi(x) \geq x$, and $\int_0^1 \frac{dx}{\psi(x)} = +\infty$ so that 
\[ |\grad W*\mu(x) - \grad W*\mu(y)|^2 \leq C^2 \psi(|x-y|^2) .\] \label{potentialestimate}
\item $1< p \leq +\infty$. \label{p ass}
\item For any $\mu, \nu, \rho$ with $\|\mu\|_{p}, \|\nu \|_p, \|\rho\|_p \leq C_p$, we have $ \|\grad W*\mu - \grad W*\nu \|_{L^2(\rho)} \leq C W_2(\mu,\nu)$. \label{Loeperestimate}
\item $W: \Rd \to [-\infty, +\infty]$ is lower semicontinuous. \label{lsc ass}
\end{enumerate}
\end{assumption}

\begin{assumption}[drift diffusion energy] \label{potential assumption} There exist $C,C'>0$ so that
\begin{enumerate}[(i)]
\item $V(x)\geq -C'$.	\label{bddbelow ass2}
\item For any $\mu \in D(\mathcal{V}_m)$, $\|\grad V\|_{L^2(\mu)} < C'$.
\item $V(x)$ is continuously differentiable and there exists a continuous, nondecreasing, concave function $\psi: [0, +\infty) \to [0, +\infty)$ satisfying $\psi(0)=0$, $\psi(x) \geq x$, and $\int_0^1 \frac{dx}{\psi(x)} = +\infty$ so that 
\[ |\grad V(x) - \grad V(y)|^2 \leq 4C^2 \psi(|x-y|^2) \] \label{potentialestimate2}
\item $1< m \leq +\infty$ \label{p ass2}.
\end{enumerate}
\end{assumption}

These assumptions ensure the following properties of the corresponding constrained interaction and drift diffusion energies:
\begin{itemize}
\item Assumption (\ref{bddbelow ass}) ensures that the energies are bounded below.
\item Assumptions (\ref{bddbelow ass}) and (\ref{lsc ass}) for the interaction energy (resp. (\ref{bddbelow ass}) and (\ref{potentialestimate}) for the drift diffusion energy) ensure energies are lower semicontinuous.
\item Assumptions (\ref{subdiffbound}), (\ref{potentialestimate}), and (\ref{Loeperestimate}) for the interaction energy (resp. (\ref{subdiffbound}) and (\ref{potentialestimate}) for the drift diffusion energy) are used to prove an ``above the tangent line'' inequality for the energies, from which we deduce they are $\omega$-convex.
\item  Assumption (\ref{p ass}) allows us to show that there exists $\tau_* >0$ so that for all $\mu \in \P_{2}(\Rd)$ and $0 \leq \tau < \tau_*$, $J_\tau(\mu) \neq \emptyset$, so that the discrete gradient flow sequence exists.
\end{itemize}

In the case of interaction energies, assumption (\ref{potentialestimate}) is a generalization of a classical potential theory result originally proved by Yudovich  in the case that $W$ is the Newtonian potential, $\mu \in L^\infty(\Rd)$, and $\psi$ is a log-Lipschitz modulus of continuity \cite[Lemma 2.1]{Yudovich}. This type of continuity assumption has appeared in  previous works on gradient flows of nonlocal interaction energies with Newtonian repulsion or attraction by Ambrosio and Serfaty \cite{AmbrosioSerfaty} and Carrillo, Lisini, and Mainini  \cite{CarrilloLisiniMainini}.

Assumption (\ref{Loeperestimate}) is a generalization of a result originally proved by Loeper in the case of Newtonian interaction of bounded densities \cite[Theorem 2.7]{Loeper}. Such an assumption has also appeared in previous works on uniqueness of Wasserstein gradient flows of interaction energies \cite{CarrilloLisiniMainini, AmbrosioSerfaty, CarrilloRosado, BertozziLaurentRosado}.

\bigskip

The main result of this section is that energies with the above properties fall within the scope of the Wasserstein gradient flow theory for $\omega$-convex functions developed in sections \ref{section 2} and \ref{section 3}, so that their gradient flows are well-posed and their discrete gradient flows converge with explicit rate.
\begin{theorem} \label{constrained interaction theorem}
Suppose $\mathcal{W}_p$ and $\mathcal{V}_m$ are given by (\ref{Wp def}) and (\ref{V def}), with $W$, $p$, $V$, and $m$ satisfying assumptions \ref{interaction assumption} and \ref{potential assumption}. Then both energies satisfy assumption \ref{main assumptions} with  $ \omega(x)= \sqrt{x \psi(x)}$,  $\lambda_\omega = 4C$, and $\tau_* = +\infty$.
\end{theorem}

\subsection{Constrained interaction via Riesz and Newtonian potentials}
Before turning to the proof of Theorem \ref{constrained interaction theorem}, we first provide a few examples of $W, p$, $V$, and $m$ that satisfy assumptions \ref{interaction assumption} and \ref{potential assumption}. In particular, if $W(x)$ is a Newtonian or Riesz potential and $p$ is sufficiently large, depending on the singularity of $W(x)$, then  $\mathcal{W}_p$ is $\omega$-convex.
\begin{proposition} \label{Wexamples prop}
The following choices of $W$ and $p$ satisfy assumption \ref{interaction assumption}
\begin{enumerate}[(a)]
\item $W(x) = \pm c |x|^{\alpha-d}$ for $2 \leq \alpha < d$, $p \geq {d/(\alpha-2)}$, $c > 0$, and $d \geq 3$, \label{first example}
\item $W(x) = c \log (|x|)$ for $p= +\infty$, $c >0$, and $d = 2$, \label{second example}
\end{enumerate}
with $C, C'>0$ depending on $\alpha$, $d,$ and $p$, and
\begin{align} \label{psidef}
\psi(x) &:= \begin{cases} x (\log x)^2 & \text{ if } 0\leq x \leq e^{-1-\sqrt{2}}, \\ x+ 2(1+\sqrt{2})e^{-1-\sqrt{2}}  & \text{ if }x > e^{-1-\sqrt{2}} . \end{cases} 
\end{align}
Furthermore, if $V(x) = W* \rho(x)$, $m = p$, and $\rho \in L^p(\Rd)$, with $W$ and $p$ as above, then $V$ and $m$ satisfy assumption \ref{potential assumption} with $C, C'>0$ depending on $\alpha$, $d$, and $p$ and $\psi(x)$ as above.
\end{proposition}

In our proof of this proposition, we will use a few well-known properties of the $p$th power of the $L^p(\Rd)$ norm as a functional on the space of probability measures:
\[ \mu \mapsto \|\mu\|_p^p := \begin{cases} \|\mu\|_p^p  &\text{if $\mu \ll \mathcal{L}^d$,} \\ +\infty & \text{otherwise} .  \end{cases} \]
For $1\leq p < +\infty,$ this functional is lower semicontinuous with respect to the weak-* topology (hence with respect to the Wasserstein metric) and convex along generalized geodesics \cite[Proposition 9.3.9]{AGS} \cite[Lemma 3.4]{McCannConvexity}. Consequently, for a generalized geodesic $\mu_\alpha$ from $\mu_0$ to $\mu_1$,
\begin{align} \label{Lp interpolation} \|\mu_\alpha\|_p \leq \max \{ \| \mu_0\|_p , \|\mu_1\|_p \} , \text{ for all } 1 \leq p < +\infty .
\end{align}
Sending $p \to +\infty$, we likewise obtain,
\begin{align} \label{Linfty interpolation} \|\mu_\alpha\|_\infty \leq \max \{ \| \mu_0\|_\infty , \|\mu_1\|_\infty \} .
\end{align}

With this, we prove  Proposition \ref{Wexamples prop}.

\begin{proof}[Proof of Proposition \ref{Wexamples prop}]
We show that $W(x)$ and $p$ satisfy assumption \ref{interaction assumption}. The result for $V(x)$ and $m$ follows analogously. Let $C = C_{\alpha, d, p}>0$, where the value of this constant may increase from line to line. Note that, in part (\ref{first example}) when $\alpha =2$ and in part (\ref{second example}), $W(x)$ is a constant multiple of the Newtonian potential. Otherwise, $W(x)$ is a constant multiple of a Riesz potential.

We begin by showing (\ref{bddbelow ass}). Let $1/p + 1/p' = 1$. Since $p \geq d/(\alpha - 2)$, $p' \leq d/(d-(\alpha-2)) < d/(d-\alpha)$, and $W(x) \in L^{p'}_\loc$. Furthermore, if $B=B_1(0)$, then $x \in \Rd \setminus B$ implies $|W^-(x)|\leq c$. Therefore,
\[ |W^-*\mu(x) | \leq |(W^- 1_{B})*\mu(x) | + |(W^- 1_{\Rd \setminus B})*\mu(x)| \leq \|W^-\|_{L^{p'}(B)} \|\mu \|_{L^p(\Rd)} + c \|\mu\|_{L^1(\Rd)} \leq C . \]

Now we show part (\ref{subdiffbound}). As before, since $p \geq d/(\alpha - 2)$, $p' \leq d/(d-(\alpha-2)) < d/(d-(\alpha-1))$, and $\grad W(x) \in L^{p'}_\loc$. We also have $x \in \Rd \setminus B$ implies $|\grad W(x)| \leq c$, so as before, $|\grad W*\mu(x)| \leq C$, which gives the result, since $\nu \in \P_2(\Rd)$.

We now prove part (\ref{potentialestimate}).
When $W(x)$ is the Newtonian potential and $p = +\infty$, this is a classical potential theory estimate (c.f. \cite[Lemma 2.1]{Yudovich},\cite[Proposition 2.1]{CarrilloLisiniMainini}) For Riesz potentials and $p \geq d/(\alpha-2)$, see Garg and Spector \cite[Theorem A]{GargSpector}.

Next, we turn to part (\ref{Loeperestimate}).When $W(x)$ is the Newtonian potential and $p=+\infty$, this inequality is due to Loeper  \cite[Theorem 2.7]{Loeper}. For the general case, we follow an argument similar to Bertozzi, Laurent, and Rosado, in their proof of uniqueness of $L^p$ solutions to the aggregation equation \cite[Theorem 3.2]{BertozziLaurentRosado}. 

  By classical singular integral estimates (c.f. \cite[Theorem 9.9]{GilbargTrudinger}, \cite[Theorem V.1]{Stein}),
\begin{align} \label{sing int est}
\|D^2 W*\mu \|_{q} \leq C_{\alpha,d, q} \|\mu\|_{r} , \quad  r := \frac{dq}{d+q(\alpha-2)} , \quad \text{for $1<q<+\infty$ and $r>1$.}
\end{align}
 Define $p_* := d/(\alpha-2)$ and $p_*' := d/(d-(\alpha - 2))$. Note that if $\|\rho\|_p \leq C_p$, we also have $\|\rho \|_{p_*} \leq C$, since $\|\rho\|_1 =1$. Then, if $\mu_\beta:= ((1-\beta)\id + \beta\bt_{\mu}^\nu)\# \mu$ is the geodesic from $\mu$ to $\nu$ ,
\begin{align*}
&\|\grad W*\mu - \grad W*\nu \|_{L^2(\rho)} \leq \|\rho\|_{p_*}^{1/2} \|\grad W*\mu - \grad W*\nu\|_{2p_*'}  = C\left\| \int_0^1 \frac{d}{d \beta} \grad W* \mu_\beta d \beta \right\|_{2p_*'} .
\end{align*}
We then apply Minkowski's integral inequality to bound this term by
\begin{align*}
& \int_0^1 \left\| \frac{d}{d \beta} \grad W* \mu_\beta\right\|_{2p_*'} d \beta = \int_0^1 \left\| D^2 W(\cdot- (1-\beta) y - \beta \bt_{\mu_0}^{\mu_1}(y) ) (\bt_{\mu_0}^{\mu_1}(y) - y) d \mu_0(y) \right\|_{2p_*'} d \beta \\
& = \int_0^1 \left\| D^2 W(\cdot- y ) (\bt_{\mu_\beta}^{\mu_1}(y) - \bt_{\mu_\beta}^{\mu_0}(y)) d \mu_\beta(y) \right\|_{2p_*'} d \beta = \int_0^1 \left\| D^2 W*[(\bt_{\mu_\beta}^{\mu_1} - \bt_{\mu_\beta}^{\mu_0})\mu_\beta] \right\|_{2p_*'} d \beta .
\end{align*}
By inequality (\ref{sing int est}) with $q= 2p_*'$ and H\"older's inequality,
\begin{align*}
\left\| D^2 W*[(\bt_{\mu_\beta}^{\mu_1} - \bt_{\mu_\beta}^{\mu_0})\mu_\beta] \right\|_{2p'} &\leq C\|(\bt_{\mu_\beta}^{\mu_1} - \bt_{\mu_\beta}^{\mu_0})\mu_\beta\|_{r} \leq  CW_2(\mu_0,\mu_1) \|\mu_\beta\|_{p_*}^{1/2} .
\end{align*} 
Therefore, combining the above inequalities, we obtain,
\begin{align*}
&\|\grad W*\mu - \grad W*\nu \|_{L^2(\rho)} \leq  C \|\mu_\beta\|_{p_*}^{1/2} W_2(\mu_0,\mu_1) \leq C W_2(\mu_0,\mu_1) , \end{align*}
where the third inequality uses (\ref{Lp interpolation}) and (\ref{Linfty interpolation}).

Finally, part (\ref{lsc ass}) follows if we adopt the convention that $W(x) = +c |x|^{\alpha - d}$ satisfies $W(0) = 0$ and for the remaining choices of $W(x)$, we take $W(0) = -\infty$.
\end{proof}

\subsection{$\omega$-convexity of constrained interaction and drift diffusion}
We now prove Theorem \ref{constrained interaction theorem}, beginning by  showing that both energies are lower semicontinuous.

\begin{proposition} \label{lsc}
The energies $\mathcal{W}_p$ and $\mathcal{V}_m$ in Theorem \ref{constrained interaction theorem} are lower semicontinuous with respect to weak-* convergence of probability measures.
\end{proposition}

\begin{proof}[Proof of Lemma \ref{lsc}]
The lower semicontinuity of $\mathcal{V}_m$ with respect to weak-* convergence is an immediate consequence of the fact $\mu \mapsto \|\mu\|_m^m$ is lower semicontinuous with respect to the weak-* topology \cite[Lemma 3.4]{McCannConvexity} and that $V(x)$ is continuous and bounded below \cite[Equation 5.1.15]{AGS}. Thus, it remains to show the result for $\mathcal{W}_p$.

First, we show that, for any $c>0$, $\{ \mu \in \P_{2,\nu_0}(\Rd): \|\mu\|_p \leq c\}$ is closed in the weak-* topology. For $1 \leq p < +\infty$, this is an immediate consequence of the fact that $\| \cdot \|_p^p$ is lower semicontinuous with respect to this topology. In the case $p=+\infty$, since $\lim_{p \to \infty} \|\mu\|_p = \|\mu\|_\infty$ and $\| \mu\|_p \leq \|\mu \|_\infty^{(p-1)/p} \|\mu\|_1^{1/p} = \|\mu \|_\infty^{(p-1)/p}$, we have
\begin{align} \label{lsclem1} \{ \mu : \|\mu\|_{L^\infty} \leq c \} = \bigcap_{1<p< +\infty} \{ \mu : \| \mu \|_p \leq c^{(p-1)/p} \} ,
\end{align}
which is also closed.

Now, we use this to prove the lower semicontinuity of $\mathcal{W}_p$. Suppose $\mu_n\wsto \mu$. If $\liminf_n \mathcal{W}_p(\mu_n) = +\infty$, the result holds trivially, so we may assume that $\liminf_n \mathcal{W}_p(\mu_n) = \tilde{C} < +\infty$. Choose a subsequence $\tilde{\mu}_n$ so that $\lim_n \mathcal{W}_p(\tilde{\mu}_n) = \tilde{C}$ and $\|\tilde{\mu}_n \|_p \leq c$. Since  $\{ \mu : \|\mu\|_p \leq c\}$ is closed, $\|\mu\|_p \leq c$. By part (\ref{bddbelow ass}) of assumption \ref{interaction assumption},  $W^-(x-y)$ is uniformly integrable with respect to $\tilde{\mu}_n \times \tilde{\mu}_n$. Therefore, by \cite[Lemma 5.1.7]{AGS},
\begin{align*}
 \liminf_{n \to \infty} \iint W(x-y) d \mu_n(x) d\mu_n(y)  =  \lim_{n \to \infty} \iint W(x-y) d \tilde{\mu}_n(x) d\tilde{\mu}_n(y)  \geq \iint W(x-y) d \mu(x) d \mu(y),
 \end{align*}
 which gives the rsult
\end{proof}

Next, we characterize the directional derivatives of the energy $\mathcal{W}_p$ and the directional derivative of the drift portion of the energy $\mathcal{V}_m$. For the latter, this is the $\omega$-convex analogue of Ambrosio, Gigli, and Savar\'e's characterization of the subdifferential for semiconvex functions \cite[Equation 10.1.7]{AGS}. For the former, our result does not offer a characterization of the subdifferential, since we only compute the directional derivative of $\mathcal{W}_p$ along curves that lie in the domain of $\mathcal{W}_p$.

\begin{proposition} \label{Econvex}
Suppose $\mathcal{W}_p$ is as in Theorem \ref{constrained interaction theorem}, $\mu_0,\mu_1 \in D(\mathcal{W}_p)$, and $\gamma  \in \Gamma(\mu_0, \mu_1)$ is such that  $\mu_\alpha(x)=( (1-\alpha) \pi_1 + \alpha \pi_2) \# \gamma$ satisfies $\|\mu_\alpha\|_p \leq C_p \text{ for all } \alpha \in [0,1]$. Then $\mathcal{W}_p(\mu_\alpha)$ is continuously differentiable, $ \left. \frac{d}{d \alpha } \mathcal{W}_p(\mu_\alpha) \right|_{\alpha = 0} =  \iint \la \grad W * \mu_{0}(x), y-x \ra d \gamma$, and for $\omega(x) := \sqrt{ x \psi(x)}$,
\[ \left|  \mathcal{W}_p(\mu_{0}) -\mathcal{W}_p(\mu_{1})  + \left. \frac{d}{d \alpha } \mathcal{W}_p(\mu_\alpha) \right|_{\alpha = 0} \right| \leq 2C \omega \left( \|x-y\|_{L^2(\gamma)}^2 \right). \]
\end{proposition}

\begin{proof}
Let $\pi_\beta = ( (1-\beta) \pi_1 + \beta \pi_2)$. First,  we compute
\begin{align} \label{heightomegaconvex1}
&\frac{d}{d \alpha} \mathcal{W}_p(\mu_\alpha) = \lim_{h \to 0} \frac{1}{h} [\mathcal{W}_p(\mu_{\alpha+h}) - \mathcal{W}_p(\mu_\alpha)]  \nonumber \\
&= \lim_{h \to 0} \frac{1}{2h} \left[ \int W*\mu_\alpha d\mu_{\alpha +h}  - \int W* \mu_{\alpha}  d \mu_\alpha \right]  +  \frac{1}{2h} \left[ \int  W* \mu_{\alpha+h} d\mu_{\alpha +h} - \int W*\mu_{\alpha+h} d\mu_{\alpha} \right] \nonumber \\
&=  \lim_{h \to 0} \frac{1}{2h} \iint [(W*\mu_\alpha)\circ \pi_{\alpha +h}   -(W* \mu_{\alpha})\circ \pi_\alpha ]d \gamma  +\frac{1}{2h} \iint  [(W* \mu_{\alpha+h})\circ \pi_{\alpha + h} - (W*\mu_{\alpha+h})\circ \pi_{\alpha} ] d\gamma  \nonumber \\
&= \lim_{h \to 0} \frac{1}{2h} \iint k_\alpha \circ \pi_{\alpha +h} - k_\alpha \circ \pi_\alpha d \gamma +\frac{1}{2h}  \iint k_{\alpha + h}\circ \pi_{\alpha +h} - k_{\alpha + h} \circ \pi_{\alpha} d \gamma , 
\end{align}
where $k_{\tilde{\alpha}}(x) := W * \mu_{\tilde{\alpha}}(x)$. We consider both terms simultaneously by taking $\tilde{\alpha} = \alpha$ or $\alpha + h$. Since $\pi_\beta(x)$ is continuously differentiable in $\beta$ and, by assumption \ref{interaction assumption} (\ref{potentialestimate}), $k_{\tilde{\alpha}}(x)$ is continuously differentiable with respect to $x$,
\begin{align*}
 k_{\tilde{\alpha}} \circ \pi_{\alpha + h} - k_{\tilde{\alpha}} \circ \pi_\alpha  &=  \int_\alpha^{\alpha +h} \frac{d}{d \beta} k_{\tilde{\alpha}} \circ \pi_\beta d\beta =  \int_{\alpha}^{\alpha +h} \la \grad W * \mu_{\tilde{\alpha}}\circ \pi_\beta, \pi_2 - \pi_1 \ra d \beta .
\end{align*}
Furthermore, since $\|\mu_{\tilde{\alpha}} \|_p, \|\mu_\beta\|_p \leq C_p$,  assumption \ref{interaction assumption} (\ref{subdiffbound}) ensures $\| \grad W * \mu_{\tilde{\alpha}} \|_{L^2(d \mu_\beta)} \leq C'$.
Consequently, we may interchange the order of integration in (\ref{heightomegaconvex1}) and add and subtract to obtain
\begin{align} \label{heightomegaconvex2}
\frac{d}{d \alpha} \mathcal{W}_p(\mu_\alpha)  = \lim_{h \to 0} \frac{1}{h}& \int_{\alpha}^{\alpha + h}  \iint_{\Rd \times \Rd}  \la \grad W * \mu_{\alpha}\circ \pi_\beta, \pi_2 - \pi_1 \ra d \gamma d \beta   \\
& + \frac{1}{2h}  \int_\alpha^{\alpha +h}\iint_{\Rd \times \Rd} \la \grad W * \mu_{\alpha +h}\circ \pi_\beta - \grad W *\mu_\alpha \circ \pi_\beta, \pi_2 - \pi_1 \ra d \gamma d \beta . \nonumber
\end{align}

In order to compute the limits on the right hand side, note that for any $\alpha, \tilde{\alpha}, \beta, \tilde{\beta} \in [0,1]$,
\begin{align} \label{heightomegaconvex3}
& \iint |\grad W *\mu_{\tilde{\alpha}} \circ \pi_{\tilde{\beta}} - \grad W* \mu_{\alpha} \circ \pi_{\beta}| |\pi_2- \pi_1| d \gamma  \\
&\quad \leq \iint \left [ |\grad W *\mu_{\tilde{\alpha}} \circ \pi_{\tilde{\beta}} - \grad W* \mu_{\tilde{\alpha}} \circ \pi_{\beta}| + |\grad W *\mu_{\tilde{\alpha}} \circ \pi_{\beta} - \grad W* \mu_{\alpha} \circ \pi_{\beta}| \right] |\pi_2- \pi_1| d \gamma \nonumber\\ 
&\quad \leq  \left[ \| \grad W *\mu_{\tilde{\alpha}} \circ \pi_{\tilde{\beta}} - \grad W * \mu_{\tilde{\alpha}} \circ \pi_\beta \|_{L^2(d \gamma)}  + \| \grad W * \mu_{\tilde{\alpha}} - \grad W *\mu_{\alpha}\|_{L^2(d \mu_\beta)} \right] \|\pi_1 - \pi_2\|_{L^2(d \gamma)} \nonumber  \\
 &\quad \leq \left[ C \sqrt{\| \psi(|\pi_{\tilde{\beta}} - \pi_{\beta}|^2) \|_{L^1(d\gamma)} } +  \| \grad W * \mu_{\tilde{\alpha}} - \grad W *\mu_{\alpha}\|_{L^2(dx)} \right] \|\pi_1 - \pi_2\|_{L^2(d \gamma)} \nonumber \\
 &\quad \leq \left[ C \sqrt{\psi(\| |\pi_{\tilde{\beta}} - \pi_{\beta}|^2  \|_{L^1(d \gamma)})} + C W_2(\mu_{\tilde{\alpha}}, \mu_\alpha) \right]  \|x-y\|_{L^2(d \gamma)} \nonumber \\
 &\quad = \left[ C\sqrt{ \psi( |\tilde{\beta} - \beta|^2 \|x-y\|_{L^2(d \gamma)}^2)} +C |\tilde{\alpha} - \alpha| \|x-y\|_{L^2(d \gamma)} \right] \|x-y\|_{L^2(d \gamma)}, \nonumber
\end{align}
where in the third inequality, we use assumption \ref{interaction assumption} (\ref{potentialestimate}), and in the fourth inequality we Jensen's inequality for the concave function $\psi(x)$ and assumption \ref{interaction assumption} (\ref{Loeperestimate}). When $\alpha = \tilde{\alpha}$, this estimate ensures
\[  \beta \mapsto \iint  \la \grad W * \mu_{\alpha}\circ \pi_\beta, \pi_2 - \pi_1 \ra d \gamma \]
is continuous, so that the first term in (\ref{heightomegaconvex2}) converges to $ \iint  \la \grad W * \mu_{\alpha}\circ \pi_\alpha, \pi_2 - \pi_1 \ra d \gamma$. Likewise, when $\beta = \tilde{\beta}$, this estimate guarantees that the second term in (\ref{heightomegaconvex2}) is bounded by
\[ \lim_{h \to 0} \frac{1}{2h} \int_\alpha^{\alpha + h} h \|x-y\|_{L^2(d \gamma)}^2 d \beta = 0 .\]
 Therefore, we conclude
\begin{align} \label{heightomegaconvex4}
 \frac{d}{d\alpha} \mathcal{W}_p(\mu_\alpha) = \iint  \la \grad W * \mu_{\alpha}\circ \pi_\alpha, \pi_2 - \pi_1 \ra d \gamma .
 \end{align}
 
By (\ref{heightomegaconvex3}) again, $\frac{d}{d\alpha} \mathcal{W}_p(\mu_\alpha)$ is continuous for $\alpha \in [0,1]$. Therefore,
\begin{align*}
 &\mathcal{W}_p(\mu_1)  =\mathcal{W}_p(\mu_0) +  \int_0^1 \frac{d}{d\alpha} \mathcal{W}_p(\mu_\alpha) d \alpha \\
&= \mathcal{W}_p(\mu_0) + \left. \frac{d}{d\alpha} \mathcal{W}_p(\mu_\alpha) \right|_{\alpha =0} + \int_0^1\iint_{\Rd \times \Rd}  \la \grad W*\mu_\alpha \circ \pi_\alpha - \grad W*  \mu_0,\pi_2 - \pi_1 \ra d \gamma  d\alpha    .
\end{align*}
To prove the result, it suffices bound the third term by $2C \omega \left(\|x-y\|_{L^2(\gamma)}^2 \right)$, where $\omega(x) = \sqrt{ x \psi(x)}$.
This follows by combining inequality (\ref{heightomegaconvex3}) with assumption \ref{interaction assumption} (\ref{potentialestimate}), which ensures $\psi(x) \geq x$,
  \begin{align*}
&\int_0^1 \left[ C\sqrt{\psi(\alpha^2 \|x-y\|^2_{L^2(d \gamma)})} + C|\alpha| \|x-y\|_{L^2(d \gamma)} \right] \|x-y\|_{L^2(d \gamma)} d \alpha \leq 2C \omega\left( \|x-y\|^2_{L^2(d \gamma)} \right).
\end{align*}
\end{proof}

\begin{proposition} \label{Vconvex}
Suppose $V(x)$ is as in Theorem \ref{constrained interaction theorem}, $\mu_0,\mu_1 \in \P_2(\Rd)$, and  $\gamma  \in \Gamma(\mu_0, \mu_1)$. Then $\int V d\mu_\alpha$ is continuously differentiable, $\left. \frac{d}{d\alpha} \int V d\mu_\alpha \right|_{\alpha =0} = \iint \la \grad V(x), y-x \ra d \gamma$, and for $\omega(x) := \sqrt{x \psi(x)}$,
\[  \int V d\mu_0 - \int V d\mu_1  + \left. \frac{d}{d \alpha} \int V d\mu_\alpha \right|_{\alpha =0} \leq 2C \omega \left(\|x-y\|_{L^2(\gamma)}^2 \right) . \]
\end{proposition}

\begin{proof}
By assumption \ref{potential assumption} (\ref{potentialestimate2}), $V(x)$ is continuously differentiable. Let $\pi_\alpha := \left( (1-\alpha)\pi_1 + \alpha \pi_2 \right)$. Then,
\begin{align*}
\frac{d}{d \alpha} \int V d\mu_\alpha = \lim_{h \to 0} \frac{1}{h} \left[ \iint V\circ\pi_{\alpha+h} d \gamma - \iint V\circ \pi_{\alpha} d \gamma \right] = \iint \la \grad V \circ \pi_{\alpha}, y-x \ra d \gamma ,
\end{align*}
is continuous in $\alpha$. Therefore,
 \begin{align*} 
\int V\mu_1&= \int V \mu_0+ \left. \frac{d}{d \alpha} \int V d\mu_\alpha \right|_{\alpha =0}+  \iint \int_0^1 \la \grad V ((1-\alpha)x +\alpha y) - \grad V(x),y-x \ra d\alpha  d \gamma
 \end{align*}
It suffices to show that the last term is bounded by $ 2C \omega \left(\| x-y\|_{L^2(\gamma)}^2 \right).$  This follows from Holder's inequality, assumption \ref{potential assumption} (\ref{potentialestimate2}), and Jensen's inequality for the concave function $\psi(x)$,
 \begin{align*}
& \iint \left| \grad V \circ \pi_\alpha -\grad V(x) \right| \left| y-x \right| d \gamma(x,y)  \leq \|\grad V \circ \pi_\alpha - \grad V \|_{L^2(d \gamma)}\|x-y\|_{L^2(\gamma)}  \\
&\quad \leq  2C \sqrt{ \| \psi(|\pi_{\alpha} - \pi_{0}|^2) \|_{L^1(d\gamma)} } \|x-y\|_{L^2(d \gamma)}  \leq 2C \sqrt{ \psi(\| |\pi_{\alpha} - \pi_{0}|^2  \|_{L^1(d \gamma)})} \|x-y\|_{L^2(d \gamma)}\\
&\quad = 2C \sqrt{ \psi(\alpha^2 \| \pi_{1} - \pi_{0}  \|^2_{L^2(d \gamma)})} \|x-y\|_{L^2(d \gamma)} \leq 2C \omega(\|x-y\|_{L^2(\gamma)}^2) .
\end{align*}
\end{proof}

We now combine these results to prove Theorem \ref{constrained interaction theorem}.

\begin{proof}[Proof of Theorem \ref{constrained interaction theorem}]
By assumption \ref{interaction assumption} (\ref{potentialestimate}) and assumption \ref{potential assumption} (\ref{potentialestimate2}), $W(x)$ and $V(x)$ are continuous, so if $\mu$ is continuous and compactly supported  $\mathcal{W}_p(\mu) < +\infty$ and $\mathcal{V}_m(\mu)< +\infty$, hence the energies are proper. Proposition \ref{lsc} ensures that both energies are lower semicontinuous with respect to weak-* convergence of probability measures, hence they are lower semicontinuous with respect to the Wasserstein metric.

We now show that for all $\mu \in \P_2(\Rd)$ and $\tau >0$, $J_\tau \mu \neq \emptyset$, so that we may take $\tau_* =+\infty$. By assumption \ref{interaction assumption} (\ref{bddbelow ass}) and assumption \ref{potential assumption} (\ref{bddbelow ass2}), the energies are bounded below, and Proposition \ref{lsc} ensures they are lower semicontinuous with respect to the weak-* topology. Therefore, by \cite[Corollary 2.2.2]{AGS}, it suffices to show that every $W_2$-bounded set contained in a sublevel of $\mathcal{W}_p$ or $\mathcal{V}_m$ is relatively compact with respect to the weak-* topology.

If $\mu_n$ is a $W_2$ bounded sequence contained in a sublevel of $\mathcal{W}_p$, then there exists $\tilde{C}>C_p$ so that $\int |x|^2 \mu_n\leq \tilde{C}$ and $\|\mu_n \|_p \leq \tilde{C}$ for all $n$. On the other hand, if $\mu_n$ is a $W_2$ bounded sequence contained in a sublevel of $\mathcal{V}$, then  there exists $\tilde{C}>0$ so that $\int |x|^2 \mu_n\leq \tilde{C}$ and $\|\mu_n \|_m \leq \tilde{C}$ for all $n$, since $V(x)$ is bounded below. The remainder of the argument is identical for both energies, so we restrict our attention to $\mathcal{W}_p$.  By assumption \ref{interaction assumption} (\ref{p ass}), $1< p \leq +\infty$, so there exists a subsequence $\tilde{\mu}_n$ and a function $\mu \in L^p(\Rd)$ so that for any $h \in L^{p'}(\Rd)$, $\int h (\tilde{\mu}_n - \mu) \to 0$. Therefore, for any bounded, continuous function $f$, if $B_R$ is the ball of radius $R>0$ centered at the origin and $1_{B_R}$ is the characteristic function on that ball,
\begin{align} \label{const int thm1} \left|  \int f  \tilde{\mu}_n - \int f  \mu \right| \leq \left| \int_{B_R} f  \tilde{\mu}_n - f \mu \right| +  \int_{B_R^c} |f| (\tilde{\mu}_n + \mu)  .
\end{align}
Since $\int |x|^2 \tilde{\mu}_n\leq \tilde{C}$,
\[ \int_{B_R^c} |f| \tilde{\mu}_n \leq \frac{\|f\|_\infty}{R^2} \int |x|^2 \tilde{\mu}_n \leq \frac{\tilde{C} \|f\|_\infty}{R^2} . \]
Therefore, we may choose $R>0$ sufficiently large, uniformly in $n$, so that the second term in (\ref{const int thm1}) is arbitrarily small uniformly in $n$. Sending $n \to \infty$ then shows that $\tilde{\mu}_n \wsto \mu$.

Finally, we show that $\mathcal{W}_p$ and $\mathcal{V}$ are $\omega$-convex along generalized geodesics, where $\omega(x)$ is an Osgood modulus of convexity and $\tomega(x) = o (\sqrt{x})$ as $x \to 0$. $\mathcal{W}_p$ is $\omega$-convex along generalized geodesics with $\lambda_\omega = 4C$ and $\omega(x) = \sqrt{ x \psi(x)}$ by Propostions \ref{omega convex sufficient} and \ref{Econvex}. Likewise, $\mathcal{V}_m$ is $\omega$-convex along generalized geodesics for $\lambda_\omega = 4C$ and $\omega(x) = \sqrt{x \psi(x)}$ since $\mu \mapsto \|\mu\|_{m}^m$ is convex and $\mu \mapsto \int V d\mu$ is $\omega$-convex by Propositions \ref{omega convex sufficient} and  \ref{Vconvex}.
 By assumptions \ref{interaction assumption} (\ref{potentialestimate}) and \ref{potential assumption} (\ref{potentialestimate2}), $\psi:[0, +\infty) \to [0, +\infty)$ is continuous, nondecreasing, and concave and $\psi(0) = 0$. Therefore, $\omega(x) = \sqrt{x \psi(x)}$ has the same properties. In particular, to see that $\omega(x)$ is concave, note that for any $x,y \geq 0$ and $\alpha \in [0,1]$,
\begin{align*}
&\omega(\alpha x + (1-\alpha) y)^2 = (\alpha x + (1-\alpha) y) \psi(\alpha x + (1-\alpha) y)  \\
&\geq   (\alpha x + (1-\alpha) y) (\alpha \psi(x) + (1-\alpha) \psi(y)) =   \alpha^2 x \psi(x) + (1-\alpha)^2 y\psi(y) + \alpha (1-\alpha) (x \psi(y)+ y\psi(x)) \\
&\geq   \alpha^2 x \psi(x) + (1-\alpha)^2 y\psi(y) + 2\alpha (1-\alpha) \sqrt{ x y \psi(x)\psi(y)} =   \left(\alpha \omega(x) + (1-\alpha) \omega(y) \right)^2 ,
\end{align*}
where the first inequality follows from the fact that $\psi(x)$ is concave and the second next inequality uses that $2ab \leq a^2 +b^2$. Consequently, by the subadditivity of concave functions,
\[|  \omega(x) - \omega(y)| \leq \omega(|x-y|) , \]
and since $\int_0^1 \frac{dx}{\omega(x)} = +\infty$, $\omega(x)$ is an Osgood modulus of continuity with $\tomega(x) = \omega(x)$. Finally, by definition of $\omega(x) = \sqrt{x \psi(x)}$ and the fact that $\lim_{x \to 0^+} \psi(x) = 0$, we have $\omega(x) = o (\sqrt{x})$ as $x \to 0$.
\end{proof}

\appendix{

\section{}

\subsection{Further examples of $\omega$-convex energies} \label{applicationappendix}
In this section, we show that any energy that is $\phi$-uniformly convex, in the sense of Carrillo, McCann, and Villani  \cite{CarrilloMcCannVillani} and differentiable along geodesics is also $\omega$-convex along  geodesics, where $\omega(x)$ is an Osgood modulus of convexity satisfying $\tilde{\omega}(x) = o(\sqrt{x})$ as $x \to 0$.

\begin{proposition} \label{Carrillo et al Prop}
Suppose $E:\P_2(\Rd) \to \R \cup \{+\infty\}$ is lower semicontinuous and, for all generalized geodesics $\mu_\alpha$ with $\mu_0, \mu_1 \in D(E)$, $E(\mu_\alpha)$ is differentiable for $\alpha \in [0,1]$ and $\frac{d}{d \alpha} E(\mu_\alpha) \in L^1([0,1])$. Then if $E$ is $\phi$-uniformly convex along generalized geodesics, in the sense of \cite[Definition 2]{CarrilloMcCannVillani}, for $\phi$ satisfying \cite[Assumptions ($\phi_0$)-($\phi_2$)]{CarrilloMcCannVillani}, $E$ is $\omega$-convex along generalized geodesics
for an Osgood modulus of convexity satisfying $\tilde{\omega}(x) = o(\sqrt{x})$ as $x \to 0$.
 Furthermore, if $E$ merely satisfies the previous assumptions along geodesics, then $E$ is $\omega$-convex along geodesics.
\end{proposition}

\begin{proof}
We prove the result for $\omega$-convexity along generalized geodesics. The proof for $\omega$-convexity along geodesics is analogous.

By \cite[Lemma 2]{CarrilloMcCannVillani}, for any generalized geodesic $\mu_\alpha$ with $\mu_0,\mu_1 \in D(E)$ and for any $\beta \in (0,1)$,
\begin{align*}
\left. \frac{d}{d\alpha} E(\mu_\alpha) \right|_{\alpha = \beta}  -  \left. \frac{d}{d \alpha} E(\mu_\alpha) \right|_{\alpha=0} \geq \phi(\beta W_{2,\bnu}(\mu_0,\mu_1))  W_{2,\bnu}(\mu_0,\mu_1) .
\end{align*}

By \cite[Assumption ($\phi_0$)]{CarrilloMcCannVillani}, either $\phi \geq 0$, in which case let $\lambda_\omega = 1$, or $\phi \leq 0$, in which case let $\lambda_\omega =-1$. Define $\omega_1(x):= \frac{2}{\lambda_\omega} \int_0^{\sqrt{x}} \phi(s) ds$ and integrate the above inequality from $\beta = 0$ to $\beta =1$ to obtain
\begin{align*}
 E(\mu_1) -E(\mu_0)  -  \left. \frac{d}{d \alpha} E(\mu_\alpha) \right|_{\alpha=0} \geq \int_0^1 \phi(\beta W_{2,\bnu}(\mu_0,\mu_1))  W_{2,\bnu}(\mu_0,\mu_1)  d \beta = \frac{\lambda_\omega}{2} \omega_1(W_{2,\bnu}^2(\mu_0,\mu_1)) . \end{align*}
 By Lemma \ref{omega convex sufficient}, this ensures that $E$ is $\omega_1$-convex along generalized geodesics.
 
We now consider the cases $\phi \geq 0 $ and $\phi \leq 0$ separately to show that $E$ is $\omega$-convex along generalized geodesics for an Osgood modulus of convexity $\omega(x)$ satisfying $\tomega(x) = o(\sqrt{x})$ as $x \to 0$. If $\phi \geq 0$, then $\lambda_\omega >0$ and the fact that $E$ is $\omega_1$-convex ensures $E$ is also $\omega$-convex for any $\omega \leq \omega_1$. In particular, the condition holds for
\begin{align*}
\omega(x):= \begin{cases}
\frac{2}{\lambda_\omega} \int_0^{\sqrt{x}} \phi(s) ds &\text{ if } 0 \leq x \leq 1 ,\\
\frac{2}{\lambda_\omega} \int_0^{1} \phi(s) ds &\text{ if } x> 1 .
\end{cases}
\end{align*}
By \cite[Assumption ($\phi_0$)]{CarrilloMcCannVillani}, $\phi(x)$ is continuous, hence $\omega(x)$ is continuous, nondecreasing, and vanishes only at zero. By \cite[Assumption ($\phi_2$)]{CarrilloMcCannVillani}, $\phi(x)$ is superadditive, hence increasing. Consequently, if $x \in (0,1]$, there exists $n \in \mathbb{N}$ so that $1/(n+1) \leq x \leq 1/n$ and
\[ \phi(x) \leq \phi \left(\frac{1}{n} \right) \leq \ \frac{1}{n} \phi(1) = \frac{1}{n+1} \frac{n+1}{n} \phi(1) \leq 2 \phi(1) x   . \]
Therefore,
\[ |\omega(x) - \omega(y)| \leq \left. \begin{cases}\frac{2}{\lambda_\omega} \left| \int_{\sqrt{x}}^{\sqrt{y}} \phi(s) ds  \right| &\text{ if } x,y \in [0,1] \\
\frac{2}{\lambda_\omega} \int_{\min\{\sqrt{x}, \sqrt{y},1 \}}^1 \phi(s) ds   &\text{ otherwise} \end{cases} \right\}  \leq \frac{2 \phi(1)}{\lambda_\omega} |y-x| . \]
Since $\tomega(x):=  \frac{2 \phi(1)}{\lambda_\omega} x $ is an Osgood modulus of continuity for $\omega(x)$ and $\tomega(x) = o(\sqrt{x})$ as $x \to 0$, this concludes our proof in the case $\phi \geq 0$.

We now consider the case $\phi \leq 0$, so $\lambda_\omega < 0$. By \cite[Assumption ($\phi_2$)]{CarrilloMcCannVillani}, $\phi(x) \geq - kx$ for some $k \in (0,+\infty)$. 
Therefore,
\[ |\omega_1(x) - \omega_1(y)| \leq  \left| \frac{2}{\lambda_\omega} \int_{\sqrt{x}}^{\sqrt{y}} \phi(s) ds  \right| \leq \frac{2}{|\lambda_\omega|} \left|  \int_{\sqrt{x}}^{\sqrt{y}} ks ds  \right| \leq \frac{k}{|\lambda_\omega|} |y-x| . \]
Since $\tomega(x) := \frac{k}{|\lambda_\omega|} x$ is an Osgood modulus of continuity for $\omega_1(x)$ and $\tomega(x) = o(\sqrt{x})$ as $x \to 0$, this concludes our proof for $\phi \leq 0$.
\end{proof}

\subsection{Properties of the ordinary differential equation associated to $\omega$} \label{ode appendix}
In this section we prove a few elementary properties of the ordinary differential equation (\ref{omega ode}) and its approximation by the explicit Euler method, which we used in our proof of the convergence of the discrete gradient flow.

\begin{proof}[Proof of Proposition \ref{odeproposition}]
The result holds trivially for $\lambda_\omega = 0$, so assume $\lambda_\omega \neq 0$.
By definition of $F_t(x)$, for any $m \in \mathbb{N}$, $h>0$, and $mh < T$,
\[ F_{mh}(x) - F_{(m-1)h}(x) =  \int_{(m-1)h}^{mh} \lambda_\omega \omega(F_s(x))ds . \]
Therefore, if $y \geq0$,
\begin{align} \label{odeprop1}
|F_{mh}(x) - f_h(y)| 
&\leq  |F_{(m-1)h}(x) - y |  + |\lambda_\omega| \int_{(m-1)h}^{mh} | \omega(F_s(x))- \omega(y) | ds
\end{align}
Consequently, if we define
\[ \bar{F}_{s,h}(x) := \ f_{h}^{(i)}(x) \text{ for }  ih \leq s < (i+1)h , \]
then, taking $y = f_h^{(m-1)}(x)$ in inequality (\ref{odeprop1}) and using that $\omega(x)$ has modulus of continuity $\tilde{\omega}(x)$,
\begin{align*}
|F_{mh}(x)-\bar{F}_{mh,h}(x)| \leq |F_{(m-1)h}(x)-\bar{F}_{(m-1)h,h}(x)| + |\lambda_\omega| \int_{(m-1)h}^{mh} \tilde{\omega}(|F_s(x)-\bar{F}_{s,h}(x)|) ds .
\end{align*}
Setting $g_h(s) := | F_s(x) - \bar{F}_{s,h}(x)|$, we may iterate this to obtain
\[ g_h(mh) \leq g_h((m-1)h) + |\lambda_\omega| \int_{(m-1)h}^{mh} \tilde{\omega}(g_h(s)) ds \leq  |\lambda_\omega| \int_0^{mh} \tilde{\omega}(g_h(s)) ds  . \]
Given $t<T$, define $m:=  \lfloor{\frac{t}{h}} \rfloor $. Then
\begin{align*} g_h(t) &= |F_t(x) - \bar{F}_{t,h}(x)| = |F_t(x) - F_{mh}(x)| + g_h(mh) \leq |\lambda_\omega| \int_{mh}^t \omega(F_s(x))ds +  |\lambda_\omega| \int_0^t \tomega(g_h(s)) ds  \\
&\leq \begin{cases} |\lambda_\omega| h \omega(F_t(x))+  |\lambda_\omega| \int_0^t \tomega(g_h(s)) ds &\text{ if } \lambda_\omega >0 , \\
|\lambda_\omega| h \omega(x)+  |\lambda_\omega| \int_0^t \tomega( g_h(s) )ds &\text{ if } \lambda_\omega <0 ,
\end{cases}
\end{align*}
where the last inequality follows since $\omega(x)$ is nondecreasing and $F_t(x)$ is either increasing in $t$ (if $\lambda_\omega >0$) or decreasing in $t$ (if $\lambda_\omega < 0$). Then Bihari's inequality \cite{Bihari} implies
\[ g_h(t) \leq \begin{cases} \tF_{|\lambda_\omega| t}(|\lambda_\omega| h \omega(F_t(x))) &\text{ if } \lambda_\omega >0 , \\
\tF_{|\lambda_\omega| t}(|\lambda_\omega| h  \omega(x)) &\text{ if } \lambda_\omega <0 .
\end{cases}. \]
Finally, taking $h = t/m$ gives the result.
\end{proof}

\begin{proof}[Proof of Lemma \ref{monotonicitylem}]
First, we show (\ref{ftau monotone}). When $\lambda_\omega \geq 0$, the result is an immediate consequence of the fact that $\omega$ is nondecreasing.  Now suppose $\lambda_\omega < 0$ and $0 \leq x \leq y \leq r$.
Since $\omega$ has an Osgood modulus of continuity $\tomega(x)$,
\begin{align} \label{ftau monotone 1} f_\tau(x) - f_\tau(y) =  (x-y) + \lambda_\omega \tau (\omega(x) - \omega(y)) \leq-|x-y| + |\lambda_\omega| \tau \tilde{\omega}(|x-y|).
\end{align}
Let $\delta = |x-y|$. If $|\lambda_\omega| \tau \tilde{\omega}(\delta) \leq \delta$, then the right hand side of (\ref{ftau monotone 1}) is negative, and (\ref{ftau monotone}) holds. Therefore, assume that $|\lambda_\omega| \tau \tilde{\omega}(\delta) \geq \delta$. Since $\sqrt{\delta} \leq |\lambda_\omega| \tau \tilde{\omega}(\delta)/\sqrt{\delta}$, we have $\delta \leq \lambda_\omega^2  c_{r}^2 \tau^2$, and since $\tilde{\omega}$ is nonincreasing, substituting this into (\ref{ftau monotone 1})  shows
\begin{align*}f_\tau(x) - f_\tau(y)  \leq |\lambda_\omega| \tau \tilde{\omega}\left( \lambda_\omega^2 c_{r}^2 \tau^2 \right) \leq \lambda_\omega^2 c_r^2 \tau^2.
\end{align*}

Now we consider (\ref{ftau break}). If $\lambda_\omega \leq 0$, by the monotonicity of $\omega$,
\[ f_\tau(x + y) = x+ y + \lambda_\omega \tau \omega(x+y) \leq x + y + \lambda_\omega \tau \omega(x) = f_\tau(x) + y . \]
If $\lambda_\omega >0$, we use the fact that $\omega$ has modulus of continuity $\tilde{\omega}$ to conclude
\[ f_\tau(x + y) = x+ y + \lambda_\omega \tau \omega(x+y) = f_\tau(x) + y + \lambda_\omega \tau (\omega(x+y) - \omega(x)) \leq f_\tau(x) + y +\lambda_\omega \tau \tilde{\omega}(y) . \]
\end{proof}

\subsection{Time-dependent energy functionals}
In this section, we describe how our previous results on the convergence of the discrete gradient flow may be adapted to prove the convergence of the discrete gradient flow of a time-dependent energy. Specifically, we consider discrete gradient flow sequences of the following form:
\begin{align} \label{timeDepJKO} \mu^n_\tau \in \argmin_{\nu \in \P_2(\Rd)} \left\{ \frac{1}{2\tau} W_2^2(\mu^{n-1}_\tau, \nu) + E^n_\tau(\nu) \right\} .
\end{align}

Similar discrete gradient flow sequences of time-dependent energies were previously studied by Petrelli and Tudorascu \cite{PetrelliTudorascu}, without explicit rates of convergence. Provided that $E^n_\tau(\nu)$ possesses sufficient continuity with respect to $n \tau$ and is uniformly $\omega$-convex as $n$ and $\tau$ vary, our previous arguments may be adapted quantify the rate of convergence in this case.

We consider energies satisfying the following assumption:
\begin{assumption}[assumption on time dependent energy] \label{time energy assumption} Suppose that $E^0_\tau = E^0_h$ for all $\tau, h>0$, and abbreviate this initial energy by $E^0$. Suppose that for any $\mu \in D(E^0)$ and $T>0$, there exists $C(\mu,T)>0$ so that 
\begin{enumerate}[(i)]
\item  \label{time ass0} $E^n_\tau$ satisfies assumption \ref{main assumptions}, with the same choice of $\omega(x)$, $\lambda_\omega$, and $\tau_*$ for all $\tau >0$, $n \in \mathbb{N}$;
\item  \label{time ass1} $E^0_\tau(\mu) - E^n_\tau(\mu^n_\tau) \leq C(\mu,T) $, $ W_2(\mu^{n-1}_\tau,\mu^n_\tau) \leq \sqrt{2 \tau C(\mu, T)}$, and $W_2(\mu , \mu^n_\tau) \leq \sqrt{2 n \tau C(\mu,T)}$ for all $0 \leq \tau< \tau_*$ and $n \in \mathbb{N}$ so that $n \tau \leq T$;
\item  \label{time ass2} there exists a continuous, nondecreasing, concave function $\sigma:[0,+\infty) \to [0, +\infty)$, vanishing only at zero so that for all $m, n \in \mathbb{N}$ and $0 \leq h \leq \tau < \tau^*$ with $mh, n \tau \leq T$,
\begin{align*}  |E^n_{\tau}(\mu^i_\tau) -E^m_{h}(\mu^i_\tau)| \leq C(\mu, T) \left( \sigma((n \tau-m h)^2) + c(\tau)  \right) &\text{ for all } 1 \leq i \leq n ,
\end{align*}
where $c(\tau) \xrightarrow{\tau \to 0} 0$.
\end{enumerate}
\end{assumption}
\noindent Item (\ref{time ass1}) is the analogue of inequalities (\ref{tau mu def}), (\ref{tau mu cons}), and (\ref{base case}) in the time varying case.

\begin{theorem}[convergence of discrete gradient flow] \label{time dep exp form thm}
Suppose $E^n_\tau$ satisfies assumption \ref{time energy assumption}. Then for all $0 \leq t \leq T$,
\begin{align*}\tF_{2t}(W_2^2(\mu^n_{t/n}, \mu(t)) &\leq \bar{C}  \left[ t/\sqrt{n} + t \tomega(\sqrt{t/n}) \right	] + 4R^2 t \left(\sigma(t^2/n)) +c(\tau)  \right) \text{ as } n \to +\infty,
\end{align*}
where $R$ and $\bar{C}$ depend on $\tomega, \lambda_{\tomega}, C(\mu,T)$, $T$, and $\tau_*$.
\end{theorem}

\begin{proof}
We consider the case $\lambda_\omega \leq 0$, since this is all that is needed to prove the above theorem.
In section \ref{one step section}, Proposition \ref{discreteEVI} remains unchanged, as it only depends on the behavior of the  energy at one step of the discrete gradient flow.
The contraction inequality, Theorem \ref{cothm}, changes slightly, since $\mu_\tau$ and $\nu_\tau$ may correspond to the energy at different times.  Suppose that 
\[ \mu_\tau \in \argmin_\rho \left\{ \frac{1}{2 \tau} W_2^2(\mu, \rho) + E^n_\tau(\nu) \right\} \quad  \text{ and } \quad \nu_\tau \in \argmin_\rho \left\{ \frac{1}{2 \tau} W_2^2(\nu, \rho) + E^m_h(\rho) \right\} . \]
Then the energies in  (\ref{discrete EVI4}) are $E^n_\tau$ and the energies in (\ref{discrete EVI5}) are $E^m_h$. Consequently, when they are added together to obtain inequality (\ref{cont1}), we must replace the first term on the right hand side by 
\begin{align} \label{time 1} 2 \tau( E^m_h(\mu) - E^n_\tau(\mu_\tau) + E^n_\tau(\nu_\tau) - E^m_h(\nu_\tau)) . 
\end{align}
Thus, for $0 \leq \tau < \min \{ 1, \tau_*, (c_r \lambda_\omega)^{-1}, C(\mu, T)/2, C(\nu, T)/2\}$, the theorem holds with the third term on the right hand side of the inequality replaced by (\ref{time 1}).

We now consider section \ref{convergence subsection}, in which we prove the convergence of the discrete gradient flow. The Asymmetric Recursive Inequality, Theorem \ref{recineqthm}, requires similar modifications as did the contraction inequality. In particular, we take $R:= \max \{ \sqrt{2(4T+1)C(\mu,T)},3 \}$. Then, as $\nu \to \nu_h$ corresponds to the energy $E^n_\tau$, we modify inequality (\ref{asym1}) by replacing the third term on the right hand side by
\begin{align*}  2 h( E^n_\tau(\mu^{m-1}) - E^m_h(\mu^n) + E^m_h(\mu^n) - E^n_\tau(\mu^n)) .
\end{align*}
Likewise, the claim continues to hold, replacing the $2 \tau (E(\mu^{m-1})-E(\mu^m))$ term by 
\begin{align*}
2 \tau (E^n_\tau(\mu^{m-1})- E^m_h(\mu^m) +  E^m_h(\mu^n) - E^n_\tau(\mu^n) ).
\end{align*}
Thus, for $m,n$ sufficiently large, the error estimate in Theorem \ref{recineqthm} becomes
\[ \bar{C}\left[ h \tomega(\sqrt{\tau}) +  h^2 + \tomega(h^2) \right] + 2 h( E^m_h(\mu^{m-1}) - E^m_h(\mu^m)) +4R^2h\left( \sigma((n \tau - mh)^2) +c(\tau)\right)  \]
with $\bar{C}$ as in (\ref{bar C1 def}).

Now, we consider Theorem \ref{W2RasBound}, quantifying the distance between two discrete gradient flows with the same initial conditions. We claim that the result continues to hold with the last term on the right hand side replaced by
\[2h(E^0(\mu) - E^{m}_h(\mu^m)) + 4R^2hm   \left(\sigma \left((n \tau - mh)^2 + \tau^2 n \right) + \sigma \left(h^2 \right) + c(\tau) \right) . \]
with $R:= \max \{ \sqrt{2(4T+1)C(\mu,T)},3 \}$. The proof proceeds as before, replacing inequality (\ref{pre exp form2}) by
\begin{align*}
&\tf^{(2(m+1))}_h (W_2^2(\mu^n_\tau, \mu^{m+1}_h))  \\
&\leq   \frac{h}{\tau} \bar{C}  \sqrt{((n-1)\tau - mh)^2 +\tau^2  (n-1)}  + \frac{\tau-h}{\tau} \bar{C}  \sqrt{ (n\tau - mh)^2+ \tau^2  n}  \nonumber \\ 
&\quad + 4R^2hm \left( \frac{h}{\tau} \sigma(((n-1)\tau - mh)^2 +\tau^2(n-1)) + \frac{\tau-h}{\tau} \sigma((n\tau - mh)^2+ \tau^2 n) + \sigma(h^2) + c(\tau) \right)\\
&\quad + \bar{C} \left[ h m \tomega(\sqrt{\tau}) + h^2m+ \tomega(h^2)m \right]+ 2h (E^0(\mu) - E_h^{m}(\mu^m)) \nonumber \\
&\quad + \bar{C} \left[ h \tomega( \sqrt{ \tau})+ h^2 + \tomega(h^2) \right] + 2h (E^{m+1}_h(\mu^{m}) - E^{m+1}_h(\mu^{m+1}))  + 4R^2 h \left(\sigma \left((n\tau - (m+1)h)^2 \right) + c(\tau) \right)
\end{align*}
By the concavity of $\sqrt{\cdot}$ and $\sigma(\cdot)$ and the fact that $ 2h(E^{m+1}_h(\mu^m) - E^m_h(\mu^m)) \leq 4Rh \sigma(h^2)$, we may bound the right hand side of the above inequality by
\begin{align*} 
&\bar{C}\sqrt{(n \tau - (m+1)h)^2 + \tau^2n} + 4R^2h(m+1) \ \left( \sigma \left((n \tau - (m+1)h)^2 + \tau^2n \right)+ \sigma \left(h^2 \right) + c(\tau) \right)\\
&+\bar{C} \left[ h(m+1) \tomega(\sqrt{\tau}) + h^2(m+1) + \tomega(h^2) (m+1) \right] + 2 h(E^0(\mu) - E^{m+1}_h(\mu^{m+1})) , \end{align*}
which gives the result.

Finally, we come to the exponential formula, Theorem \ref{expform}. Taking $\tau = t/n$ and $h = t/m$ in Theorem \ref{W2RasBound}, with the modifications described above, we obtain
\begin{align*} \tf^{(2m)}_h(W_2^2(\mu^n_{t/n}, \mu^m_{t/m})) &\leq \bar{C} \left[ t/\sqrt{n} + t \tomega(\sqrt{t/n}) + t^2/m + \tomega(t^2/m^2)m \right] + 2(t/m)(E^0(\mu) - E^m_h(\mu^m)) \nonumber \\
&\quad+ 4Rt \left(\sigma(t^2/n) + \sigma(t^2/m^2) + c(\tau) \right) .
\end{align*}
Consequently, Proposition \ref{odeproposition} and item (\ref{time ass1}) ensure
\begin{align} \label{exp form time} \tF_{2t}(W_2^2(\mu^n_{t/n}, \mu^m_{t/m})) &\leq \bar{C} \left[ t/\sqrt{n} + t \tomega(\sqrt{t/n}) + t^2/m + \tomega(t^2/m^2)m \right]  + R^2(t/m) \nonumber \\
&\quad + 4R^2 t \left(\sigma(t^2/n)+ \sigma(t^2/m^2) + c(\tau) )\right) + \tF_{\lambda_\omega^- 2t} \left(\frac{\lambda_\omega^- 2 t \tomega(R^2)}{m} \right) .
\end{align}
Therefore, $\mu^n_{t/n}$ is Cauchy, hence converges to a limit $\mu(t)$. Sending $m \to +\infty$ in (\ref{exp form time}) gives the result.

\end{proof}

\noindent {\bf Acknowledgments.} The author would like to thank Luigi Ambrosio, Eric Carlen, Nestor Guillen, Inwon Kim, Guiseppe Savar\'e, and Yao Yao for useful discussions.

\bibliographystyle{spmpsci}
\bibliography{interaction_energy}

\end{document}